\documentclass[12pt,intlimits]{amsart}

\usepackage{amsthm,amsmath,amssymb,amsopn,amsfonts,amscd,amsbsy,amsrefs}
\usepackage{mathrsfs}
\usepackage{mathtools}
\usepackage{url} 
\usepackage{fullpage}
\makeatletter
\@namedef{subjclassname@2020}{%
	\textup{2020} Mathematics Subject Classification}
\makeatother

\newtheorem{theorem}{Theorem}[section]
\newtheorem{corollary}[theorem]{Corollary}
\newtheorem{lemma}[theorem]{Lemma}
\newtheorem{proposition}[theorem]{Proposition}

\theoremstyle{Theorem}

\newtheorem{remark}[theorem]{Remark}

\numberwithin{equation}{section}

\newcommand\numberthis{\addtocounter{equation}{1}\tag{\theequation}}

\DeclareMathOperator\re{Re}
\DeclareMathOperator{\ord}{ord}
\DeclareMathOperator{\Gal}{Gal}
\DeclareMathOperator{\Esp}{\mathrm{E}}
\DeclareRobustCommand{\stirling}{\genfrac\{\}{0pt}{}} 
\newcommand{\Fq}{\mathbb{F}_q}
\newcommand{\Fp}{\mathbb{F}_p}

\newcommand{\Q}{\mathbb{Q}}

\newcommand{\R}{\mathbb{R}}
\newcommand{\C}{\mathbb{C}}

\newcommand{\calH}{\mathcal{H}}
\newcommand{\calL}{\mathcal{L}}
\newcommand{\calM}{\mathcal{M}}
\newcommand{\calO}{\mathcal{O}}
\newcommand{\calP}{\mathcal{P}}

\newcommand{\rand}{\mathrm{rand}} 

\newcommand{\prob}{\mathbb{P}} 
\newcommand{\pbx}{\mathbb{X}} 
\newcommand{\pby}{\mathbb{Y}} 
\newcommand{\pbw}{\mathbb{W}} 

\newcommand{\Tag}{\mathrm{Tag}}
\newcommand{\St}{\mathrm{St}}

\begin{document}

	\baselineskip=17pt
	
	
	\title{Quadratic Euler-Kronecker constants in positive characteristic}
	\author{Amir Akbary}
	\address{Department of Mathematics and Computer Science, University of Lethbridge, 4401 University Drive West, Lethbridge, T1K 3M4, Alberta, Canada}
	\email{amir.akbary@uleth.ca}
	\author{F\'elix Baril Boudreau}
	\address{Département de Mathématiques, Université du Luxembourg, Maison du nombre, 6, avenue de la Fonte, L-4364 Esch-sur-Alzette}
	\email{felix.barilboudreau@uni.lu}
	\thanks{Amir Akbary was supported by the NSERC discovery grant RGPIN-2021-02952. The project started when F\'elix Baril Boudreau was supported by a PIMS postdoctoral fellowship at the University of Lethbridge. This research was also funded in part by the Luxembourg National Research Fund (FNR) grant reference 501100001866-17921905. For the purpose of open access, and in fulfilment of the obligations arising from the grant agreement, the authors have applied a Creative Commons Attribution 4.0 International (CC BY 4.0) license to any Author Accepted Manuscript version arising from this submission.}
	
	\begin{abstract}
		\par In 2006, Ihara defined and systematically studied a generalization of the Euler-Mascheroni constant for all global fields, named the Euler-Kronecker constants. This paper examines their distribution across geometric quadratic extensions of a rational global function field, via the values of logarithmic derivatives of Dirichlet $L$-functions at 1.  Using a probabilistic model, we show that the values converge to a limiting distribution with a smooth, positive density function, as the genii of quadratic fields approach infinity. We then prove a discrepancy theorem for the convergence of the frequency of these values, and obtain information about the proportion of the small values. Finally, we prove omega results on the extreme values. Our theorems imply new distribution results on the stable Taguchi heights and logarithmic Weil heights of rank 2 Drinfeld modules with CM.
	\end{abstract}

	
	
	\maketitle

	\section{Introduction}
	
In \cite{ihara_2006}, Ihara defines the Euler-Kronecker constant of a global field $F$ as the limit $$
\gamma_F = \lim_{s \to 1}\left( \frac{\zeta'(s,F)}{\zeta(s,F)} + \frac{1}{s-1} \right), $$
where $\zeta(s,F)$ is the Dedekind zeta function of $F$ in the complex variable $s$ (see also \cite{Hashimoto_Iijima_Kurokawa_Masato_2004}*{pp.496-497} for the number field case). If $F$ is a global function field, i.e., $F$ is the function field of a proper, smooth, and geometrically connected curve $C$ defined over a finite field $\Fq$,  it is then shown in \cite{ihara_2006}*{(1.4.3)}) that \begin{equation}\label{Ihara_Rational_Points}  \frac{\gamma_F}{\log{q}}= \sum_{m=1}^{\infty} \left(\frac{q^m+1-N_m}{q^m}\right)+\frac{q-3}{2(q-1)}, \end{equation}  where $N_m$ denotes the number of $\mathbb{F}_{q^m}$-rational points of $C$. This formula allows explicit computation of $\gamma_F$. For instance, when $F = \Fq(t)$, then $N_m = q^m + 1$ for all $m \geq 1$, hence $\gamma_q := \gamma_F = (q-3)\log(q)/(2(q-1))$. Now, consider a supersingular elliptic curve $E/\Fp$ for a prime $p \geq 5$ with Weierstrass equation $y^2 = D$ for some monic $D$ in $\Fp[t]$. For each integer $m \geq 1$, $N_m$ equals $p^m + 1$ if $m$ is odd, and $\left(p^{\frac{m}{2}} - (-1)^{\frac{m}{2}}\right)^2$ if $m$ is even (see \cite{silverman_2009}*{Exercise 5.15}). Therefore, if $F$ is the function field of $E/\mathbb{F}_p$, then $$   \gamma_F = \left( \sum_{m=1}^\infty \left( \frac{p^m + 1 - N_m}{p^m}\right) + 1 - \frac{p+1}{2(p-1)} \right) \log(p) = \frac{p-5}{2(p+1)}.
$$

More generally, Ihara observes from \eqref{Ihara_Rational_Points} that if $N_m$ is large for small $m$ (in particular for $m = 1$), then $\gamma_F$ tends to be negative (see also \cite{ihara_2006}*{Theorem 4}). Ihara puts forward $\gamma_F$ as an invariant of $F$ and investigates its cardinality. If $F$ has genus $g_F$, then one has (see \cite{ihara_2006}*{Theorem 1 and Proposition 3})
		\begin{equation}\label{bound1}
			-\log\left( q^{g_F-1}\right)+O_q(1) \leq \gamma_F\leq 2\log\left(\log\left(q^{g_F-1}\right)\right)+O_q(1).
	\end{equation}
Moreover, for an infinite family of global function fields $F$ of unbounded genii $g_F$ and fixed degree $N \geq 1$ over a given rational function field $\Fq(t)$, Ihara shows \cite{ihara_2006}*{Corollary 2}
		\begin{equation}
			\label{bound2}
			\gamma_F>-2(N-1+\varepsilon)\log\left(\log\left(q^{g_F-1}\right)\right),
		\end{equation}
		for any real number $\varepsilon>0$.
	
	Let $q$ be a power of an odd prime number. Let $\Fq[t]$ be the ring of polynomials in the variable $t$ with coefficients in $\Fq$. For each integer $n \geq 1$, let $\calH_n$ be the subset  of degree $n$ square-free monic polynomials in $\Fq[t]$. In this paper, we study the magnitude of $\gamma_D $, the Euler-Kronecker constant of $K_D := \Fq(t)(\sqrt{D})$, the function field of the hyperelliptic curve of affine model $y^2=D$ as $D$ varies in the set $\calH_n$. From the Hurwitz genus formula, the genus of $K_D$ is
	\begin{equation}\label{Equation_Genus_gD}
		g_D = \begin{cases}
			\frac{n}{2} - 1 &\text{ if } n \text{ is even},\\
			\frac{n-1}{2} &\text{ if } n \text{ is odd}.
		\end{cases}
	\end{equation}
	
	Now, let $\chi_D$ be the quadratic Dirichlet character associated with a given $D \in \calH_n$ and let $L(s,\chi_D)$ be the corresponding Dirichlet $L$-function. Its logarithmic derivative at $s = 1$ satisfies the relation
	\begin{equation}\label{Quadratic_Euler_Kronecker-0}
		\gamma_D=\frac{L'(1,\chi_D)}{L(1,\chi_D)} + \widetilde{\gamma_q},   \end{equation}   where  $$
	\widetilde{\gamma_q} =    \begin{cases}     \gamma_q &\text{if } n \text{ is odd},\\     \gamma_q+\zeta_q(2)&\text{if } n \text{ is even},\\    \end{cases} $$
	with $\zeta_q(2) := \zeta(2,\Fq[t]) = q/(q-1)$. (See Subsection \ref{Projective_Setting} for a proof of \eqref{Quadratic_Euler_Kronecker-0}.) As we see from \eqref{Quadratic_Euler_Kronecker-0}, the magnitudes of $\gamma_D$ and $L'(1,\chi_D)/L(1,\chi_D)$ are intimately related. From \cite{ihara_2008}*{(6.8.20)}, we have
	\begin{equation}
		\label{upper bound}
		\left| \frac{L'(1,\chi_D)}{L(1,\chi_D)}\right|\ll_q \log\left( \log{q^n} \right),
	\end{equation}
	and so $\left|\gamma_D \right|\ll_q\log\left( \log{q^n} \right)$.

	In this paper, we show that for a proportion of polynomials $D$ in $\calH_n$ the values of $| {L'(1,\chi_D)}/{L(1,\chi_D)}|$ are significantly smaller than the upper bound in \eqref{upper bound}.
	
	\begin{theorem}\label{Theorem_Application_Small_Absolute_Value}
		For $n \geq 1$, we have
		$$
		\min\limits_{D \in \calH_n} \left( \left|\frac{L'(1,\chi_D)}{L(1,\chi_D)}\right| \right) \ll_q \frac{ \log^2(\log(q^n))}{\log(q^n)}.
		$$
		More precisely, there are $\gg_q q^n \log^2(\log(q^n))/\log(q^n)$ polynomials $D \in \calH_n$ for which
		$$
		\left| \gamma_D-\widetilde{\gamma}_q\right|= \left |\frac{L'(1,\chi_D)}{L(1,\chi_D)} \right|\ll_q \frac{ \log^2(\log(q^n))}{\log(q^n)}.
		$$
	\end{theorem}

	We also obtain unconditional omega results which imply that Ihara's upper bound \eqref{bound1} and lower bound \eqref{bound2} for $\gamma_D$ cannot be improved to $\mid \gamma_D \mid \leq  \log(\log(q^{g_D})) + O_q(1)$.
	\begin{theorem}\label{Theorem_Omega_Result}
	 	Let $\calP_n$ be the set of irreducible elements in $\calH_n$. For any $\varepsilon > 0$ and all large $n$, there are $\gg q^{\frac{n}{2}}$ elements $Q \in \calP_n$ such that
	 	$$
	 	\frac{L'(1,\chi_Q)}{L(1,\chi_Q)} \geq \log(\log(q^n))  + \log(\log(\log(q^n))) - A_q - \varepsilon,
	 	$$
	 	and $\gg q^{\frac{n}{2}}$ elements $Q \in \calP_n$ such that
	 	$$
	 	\frac{L'(1,\chi_Q)}{L(1,\chi_Q)} \leq -\log(\log(q^n)) - \log(\log(\log(q^n))) + A_q + \varepsilon,
	 	$$
	 	where $A_q := (2.61)\log(q) + \log(2\log(2)\zeta_q(2)) > 0$.
	 	
	 	Bounds for $\gamma_D$ are obtained by replacing $L'(1,\chi_Q)/L(1,\chi_Q)$ with $\gamma_D$ and $A_q$ with $A_q + \tilde{\gamma}$.
	\end{theorem}
	
	We prove Theorem \ref{Theorem_Application_Small_Absolute_Value} following the number field method in \cite {akbary_hamieh_2024}, which drew ideas from \cite{lamzouri_2015}, \cite{lamzouri_lester_radziwill_2019}, \cite {lamzouri_languasco_2023} and \cite{hamieh_mcclenagan_2022}. More precisely, we establish an asymptotic formula for the $r$-th integral moments of $-L’(1,\chi_D)/L(1,\chi_D)$ (Proposition \ref{Proposition_Average_Power_LprimeL}), uniform on a certain range for $r$, and use it to compute the Laplace transform of $-L’(1,\chi_D)/L(1,\chi_D)$ asymptotically (Proposition \ref{Proposition_Estimation_Exponential_Series}). We then introduce a random series $-\frac{L'(1,\pbx)}{L(1,\pbx)}$ as the global function field analogue of the model in \cite{Granville_Soundararajan_2003} for integers, and compare its Laplace transform to that of $-L'(1,\chi_D)/L(1,\chi_D)$ (Proposition \ref{Proposition_Probabilistic_Estimation_Exponential_Series}). This yields, together with an application of Berry-Esseen inequality, a discrepancy estimate between the characteristic function of the series $-\frac{L'(1,\pbx)}{L(1,\pbx)}$ and the sequence of characteristic functions of some arithmetic functions (Proposition \ref{Proposition_Discrepancy}). The proof of Theorem  \ref{Theorem_Application_Small_Absolute_Value} results from the positivity of the density function associated with the distribution of the random series (Proposition \ref{Proposition_Positive_Density}) following an argument analogous to \cite{lamzouri_languasco_2023}*{p.368} (see also \cite{akbary_hamieh_2024}*{Proof of Theorem 1.5} and \cite{hamieh_mcclenagan_2022}*{Corollary 1.4}).
	
	Proposition 3.2, a crucial result in the proof of Theorem 1.1, provides an approximation for the value $-L'(s, \chi)/L(s, \chi)$ for a non-trivial Dirichlet character $\chi$ and $\re(s) > 1/2$ in terms of a short Dirichlet polynomial. Although we only need such a result for $s=1$ and a quadratic Dirichlet character, we prove the general statement as a worthwhile addition to the literature. The proof of this key assertion substantially differs from its number field analogue. We exploit the fact that $L(s, \chi)$ is a polynomial in $q^{-s}$ with roots $\alpha_i(\chi)$ satisfying $|\alpha_i(\chi)| \leq q^{1/2}$ by the proven Riemann Hypothesis for global function fields. Consequently, we establish an optimal error term in Proposition 3.2, which gives better error terms and more flexible range of moments for global function field counterparts of some of the assertions of \cite{lamzouri_2015} (see Propositions \ref{Proposition_Average_Power_LprimeL}, \ref{Proposition_Estimation_Exponential_Series}, \ref{Proposition_Probabilistic_Estimation_Exponential_Series}). 
	
	We prove Theorem \ref{Theorem_Omega_Result} using Granville--Soundararajan's strategy \cite{Granville_Soundararajan_2003}*{Proposition 9.1 and Theorem 5a}. Others have used this method to obtain analogous results for $L'(1,\chi_D)/L(1,\chi_D)$ in the number field case (\cite{Mourtada_Murty_2013}*{Theorem 2}), and for $L(1,\chi_D)$ over global function fields (\cite{lumley_2019}*{Theorem 1.6}). Our result, like \cite{lumley_2019}*{Theorem 1.6}, is unconditional, unlike those of \cite{Granville_Soundararajan_2003}*{Proposition 9.1 and Theorem 5a} and \cite{Mourtada_Murty_2013}*{Theorem 2}, which require the Generalized Riemann Hypothesis. Our Theorem \ref{Theorem_Omega_Result} stands alongside Granville--Soundararajan's, as omega results with explicit numerical bounds.

	\subsection*{Applications to Taguchi and Weil Heights}
	
	In this subsection, we use Theorems \ref{Theorem_Application_Small_Absolute_Value} and \ref{Theorem_Omega_Result} to provide refined information on stable Taguchi heights of Drinfeld modules and logarithmic Weil heights of their $j$-invariants. We start with background material and refer the reader to \cite{papikian_2023} and \cite{wei_2017} for more details.
	
	Let $F$ be a field containing $\Fq$ and $F\{\tau\}$ the set of polynomials in the variable $\tau$ and coefficients in $F$. Endowed with the usual polynomial addition and the twisted multiplication $\tau a = a^q \tau$ for all $a \in F$, the set $F\{\tau\}$ is a non-commutative ring. 
	
	Let $|\cdot|$ be the absolute value corresponding to the place $\infty$ of $\Fq(t)$, normalized by $|t|=q$. This absolute value extends uniquely to an algebraic closure $\overline{\Fq(t)_\infty}$ of the completion $\Fq(t)_\infty$ of $\Fq(t)$ at $\infty$. We write $\C_\infty$ for the algebraically closed complete valued field of characteristic $p > 0$ that is the completion of $\overline{\Fq(t)_\infty}$. We embed $\Fq(t)$ into $\C_\infty$ via $\infty$. For $F$ a subfield of $\C_\infty$, we let $\overline{F}$ be its algebraic closure in $\C_\infty$ and $F^\text{sep}$ its separable closure in $\overline{F}$.
	
	A Drinfeld $\Fq[t]$-module over $\C_\infty$ of rank $2$ is a ring morphism $\rho: \Fq[t] \to \C_\infty\{\tau\}$ with
	\begin{equation}\label{Defining_Equation_Rank_2_Drinfeld_Module}
	\rho(t) = t + a \tau + \Delta \tau^2,
	\end{equation}
	where $a, \Delta \in \C_\infty$ and $\Delta \neq 0$. We simply say Drinfeld module if there is no risk of confusion. Its $j$-invariant is the element $j(\rho) := a^{q+1}/\Delta$ of $\C_\infty$. If $F$ is a subfield of $\C_\infty$ such that $\rho(\Fq[t])$ is in $F\{\tau\}$, then we say that $\rho$ is defined over $F$ and denote by $\rho/F$. In particular, $\rho/\overline{\Fq(t)}$ is defined over a finite extension of $\Fq(t)$ since $a$ and $\Delta$ are algebraic over $\Fq(t)$.
	
	A morphism $\rho_1 \to \rho_2$ of Drinfeld modules over $F$ is a $\phi \in F\{\tau\}$ such that $\phi \rho_1(a) = \rho_2(a) \phi$ for all $a \in \Fq[t]$. In particular, if $\phi$ is invertible, we say it is an isomorphism and $\rho_1$ and that $\rho_2$ are isomorphic over $F$, while if $\rho_1 = \rho_2$, we say $\phi$ is an endomorphism of $\rho/F$. That $\rho_1$ and $\rho_2$ are isomorphic over $F^\text{sep}$ amounts to $j(\rho_1) = j(\rho_2)$ (see \cite{papikian_2023}*{Lemma 3.8.4}). If $\rho/F$ is a Drinfeld module given by \eqref{Defining_Equation_Rank_2_Drinfeld_Module}, then up to an $F^\text{sep}$-isomorphism we can and will consider it as given by $\rho(t) = t + \tau + j(\rho)^{-1}\tau^2$.
	
	Now suppose $F/\Fq(t)$ is a finite extension. We can define the logarithmic Weil height $h(j(\rho))$ of $j(\rho)$ (see \cite{Angles_Armana_Bosser_Pazuki_2024}*{(3)}), and, related to it, a naive height for $\rho/F$, as in \cite{David_Denis_1999}*{D\'efinition 2.5}, by
	$$
	H(\rho) := \max\{h(t),h(1),h(j(\rho)^{-1})\} = \max\{1,h(j(\rho))\} \leq 1 + h(j(\rho)).
	$$
	
	Another height, $h_\Tag(\rho/F)$, is defined for $\rho/F$ by Taguchi in \cite{taguchi_1993}*{Section 5} and coined by Wei in \cite{wei_2017} as the Taguchi height of $\rho/F$. This height plays the role of an analogue of the Faltings height introduced in \cite{faltings_1983}. Up to a finite extension of $F$ of  degree at most $q^3$, $\rho/F$ has ``stable reduction everywhere'' (see \cite{David_Denis_1999}*{Lemme 2.10}). From \cite{wei_2017}*{p.1066}, it follows that the Taguchi heights $h_\Tag(\rho/L)$ coincide for all large enough extensions $L/F$. Wei then introduces the stable Taguchi height of $\rho$ as the always-existing limit
	$$
	h^{\mathrm{St}}_\Tag(\rho) := \log(q)\lim_{F'/F \text{ finite }} h_\Tag(\rho/F').
	$$
	That $\rho_1/\overline{\Fq(t)}$ and $\rho_2/\overline{\Fq(t)}$ are isomorphic over $\overline{\Fq(t)}$ amounts to $h^{\mathrm{St}}_\Tag(\rho_1) = h^{\mathrm{St}}_\Tag(\rho_2)$ (see \cite{wei_2017}*{Remark 4.2 (3)}). For $\rho/\overline{\Fq(t)}$, \cite{David_Denis_1999}*{Lemme 2.14 (v)} gives $h_\Tag^\St(\rho) \leq 5 H(\rho) + 1$ and so
	\begin{equation}\label{Bridge_Taguchi_Weil}  h(j(\rho)) \geq \frac{1}{5} h_\Tag^\St(\rho) - \frac{6}{5}.
	\end{equation}
	
	The set $\mathrm{End}_{\Fq[t]}(\rho/F)$ of endomorphisms of $\rho/F$ is a ring. For an odd integer $n \geq 1$ and $D \in \calH_n$, we let $K_D$ be the quadratic extension of $\Fq(t)$ generated by a square root of $D$ in $\overline{\Fq(t)}$  (a so-called imaginary quadratic field). The integral closure of $\Fq[t]$ in $K_D$ is $\calO_D := \Fq[t][\sqrt{D}]$. A Drinfeld module $\rho/\overline{\Fq(t)}$ is said to have complex multiplication by $\calO_D$ (or CM by $\calO_D$, for short) if the rings $\mathrm{End}_{\Fq[t]}(\rho/\overline{\Fq(t)})$ and $\calO_D$ are isomorphic.
	
	Next, we discuss connections between these heights and Theorems \ref{Theorem_Application_Small_Absolute_Value} and \ref{Theorem_Omega_Result}. Colmez gives in \cite{colmez_1993}*{Th\'eor\`eme 0.3 and Conjecture 0.4} a conjectural geometric Chowla-Selberg formula that connects the Faltings height of CM Abelian varieties over $\Q$, some special $\Gamma$-values and the value at $s = 0$ of logarithmic derivatives of certain Artin $L$-functions. His conjecture was extended in \cite{obus_2013} and in \cite{Andreatta_Eyal_Howard_Madapusi-Pera_2018} and \cite{yuan_zhang_2018} the authors prove an average version of the conjecture. In \cite{wei_2017}*{Theorem 4.4} and \cite{wei_2020}*{Theorem 1.6}, Wei establishes a Drinfeld-module analogue of Colmez's formula. We now follow and build on his \cite{wei_2017}*{Remark 4.6 (ii)}. In the context of this section, Wei establishes \begin{equation}\label{Equation_Colmez-type_formula}  h^{\mathrm{St}}_\Tag(\rho) = \frac{\log(q^n)}{4}  - \frac{q\log(q)}{2(q-1)} + \frac{1}{2h_D} \sum_{f \in \calM_{\leq n-1}} \chi_D(f) \log\left( \frac{N(f)}{q^n} \right),
\end{equation}
where $h_D$ is the class number of $\calO_D$ and $\calM_{\leq n-1}$ is the set of monic polynomials of $\Fq[t]$ of degree at most $n-1$. Wei's Chowla-Selberg formula reads \begin{equation}\label{Equation_Chowla-Selberg}   \frac{L'(0,\chi_D)}{L(0,\chi_D)} = -\log(q^n) - \frac{1}{h_D} \sum_{f \in \calM_{\leq n-1}} \chi_D(f) \log\left( \frac{N(f)}{q^n} \right),  \end{equation}  
Substituting \eqref{Equation_Chowla-Selberg} into \eqref{Equation_Colmez-type_formula} yields \begin{equation}\label{Rewritten_Colmez-type_formula} h^{\mathrm{St}}_\Tag(\rho) = -\frac{\log(q^n)}{4} -\frac{q\log(q)} {2(q-1)}-\frac {1} {2}\frac{L’(0,\chi_D)}{L(0,\chi_D)},
\end{equation}
which is closely related to Wei's Colmez-type formula (see \cite{wei_2017}*{Corollary 1.2} for the formula). Deriving the functional equation $L(s,\chi_D) = \left( q^{1-2s} \right)^{\frac{n-1}{2}} L(1-s,\chi_D)$ on $\re(s) \neq 1/2$ gives \begin{equation}\label{Equation_Relation_Log-Derivatives}   \frac{L'(s,\chi_D)}{L(s,\chi_D)} = -\log(q^{n-1}) - \frac{L'(1-s.\chi_D)}{L(1-s,\chi_D)}.
\end{equation}
In particular, taking $s = 1$ in \eqref{Equation_Relation_Log-Derivatives}, we can rewrite \eqref{Rewritten_Colmez-type_formula} as  \begin{equation}\label{Final_Colmez-type_formula}   h^{\mathrm{St}}_\Tag(\rho) = \frac{\log(q^n)}{4} - \frac{(2q-1)\log(q)}{q-1} + \frac{1}{2}\frac{L'(1,\chi_D)}{L(1,\chi_D)}.
\end{equation}
For large odd $n$, Wei obtains from \cite{ihara_2006}*{upper bound (0.6) and lower bound (0.12)} \begin{equation}\label{BigO_TaguchiHeights}  h^\mathrm{St}_\Tag(\rho) = \frac{\log(q^n)}{4} + O_q\left( \log(\log(q^n))\right).
\end{equation}
By \cite{ihara_2006}*{Theorem 1 and Corollary 2}, \eqref{Equation_Genus_gD} and \eqref{Final_Colmez-type_formula}, we can make a more precise statement than \eqref{BigO_TaguchiHeights}: For large odd $n$ and $\varepsilon > 0$ we have $|\gamma_D| < 2(1+\varepsilon)\log(\log(q^\frac{n-3}{2}))$ and thus \begin{equation}\label{Epsilon_Taguchi_Bounds}   \left| h^\mathrm{St}_\Tag(\rho) - \frac{\log(q^n)}{4} + \frac{(9q-7)\log(q)}{4(q-1)} \right| < (1+\varepsilon)\log(\log(q^{\frac{n-3}{2}})).
\end{equation}
	
From the equation after \cite{Angles_Armana_Bosser_Pazuki_2024}*{(26)}, together with \eqref{Equation_Genus_gD} and \eqref{Bridge_Taguchi_Weil}, we get
\begin{equation}\label{A_A_B_P_Log_Weil_Height}   h(j(\rho)) \geq \left( \frac{1}{20} - \frac{1}{10(q^{1/2}+1)} \right)\log(q^n) - \frac{3q - q^{1/2} -1}{10(q-1)}\log(q) - \frac{6}{5}.
\end{equation}
In our context, \eqref{A_A_B_P_Log_Weil_Height} is sharper than the general lower bound given in \cite{Angles_Armana_Bosser_Pazuki_2024}*{Proposition 4.18}.

Parts (i), (ii) and (iii) of Corollary \ref{Corollary_Stable_Taguchi_Height} respectively refine \eqref{BigO_TaguchiHeights}, \eqref{Epsilon_Taguchi_Bounds} and \eqref{A_A_B_P_Log_Weil_Height}. It follows from \eqref{Final_Colmez-type_formula} and Theorems \ref{Theorem_Application_Small_Absolute_Value} and \ref{Theorem_Omega_Result}.
	\begin{corollary}\label{Corollary_Stable_Taguchi_Height}
		\begin{enumerate}
			\item[(i)] For any odd $n \geq 1$, there are $\gg q^n \log^2(\log(q^n))/\log(q^n)$ polynomials $D$ in $\calH_n$ such that, for each such $D$, each $\rho/\overline{\Fq(t)}$ of rank $2$ with CM by $\calO_D$ satisfies
			$$
			 h^{\mathrm{St}}_\Tag(\rho) = \frac{\log(q^n)}{4} - \frac{(2q-1)\log(q)}{q-1} + O_q\left(\frac{\log^2(\log(q^n))}{\log(q^n)}\right).
			$$
			\item[(ii)] For any $\varepsilon > 0$ and odd large $n \geq 1$, 
			there are $\gg q^{n/2}$ polynomials $Q$ in $\calP_n$ such that for each of such $Q$, $\rho/\overline{\Fq(t)}$ of rank $2$ with CM by $\calO_D$ satisfies
			$$
			h^{\mathrm{St}}_\Tag(\rho) \geq \frac{\log(q^n)}{4} + \frac{\log(\log(q^n))}{2} + \frac{\log(\log(\log(q^n)))}{2} - \frac{(2q-1)\log(q)}{q-1} - \frac{A_q + \varepsilon}{2},
			$$
			and $\gg q^{n/2}$ polynomials $Q$ in $\calP_n$ such that for each of such $Q$, each $\rho/\overline{\Fq(t)}$ of rank $2$ with CM by $\calO_D$ satisfies
			$$
			h^{\mathrm{St}}_\Tag(\rho) \leq \frac{\log(q^n)}{4} - \frac{\log(\log(q^n))}{2} - \frac{\log(\log(\log(q^n)))}{2} - \frac{(2q-1)\log(q)}{q-1} + \frac{A_q + \varepsilon}{2}.
			$$
			\item[(iii)] The Drinfeld $\Fq[t]$-modules appearing in the first part of (ii) also satisfy
			$$
			h(j(\rho)) \geq \frac{\log(q^n)}{20}  + \frac{\log(\log(q^n))}{10} + \frac{\log(\log(\log(q^n)))}{10} - \frac{(2q-1)\log(q)}{5(q-1)} - \frac{12 + A_q + \varepsilon}{10}.
			$$
		\end{enumerate}
	\end{corollary}
	
\section{Notation and Prerequisites}
	Let $\R$ denote the field of real numbers, $\log$ the natural logarithm, and $m$, $n$ positive integers. We write $\lfloor x \rfloor$ for the greatest integer smaller than a real number $x$. The cardinality of a set $S$ is denoted by $\# S$. For real-valued functions $A$ and $B$ defined on a set $X$, we write $A(x) = O(B(x))$ or $A(x) \ll B(x)$ to say that there is a $C> 0$ for which $|A(x)| \leq C |B(x)|$ for all $x \in X$. The notation $A(x) \asymp B(x)$ means that $A(x) \ll B(x)$ and $B(x) \ll A(x)$. A constant with a subscript indicates that the constant depends on that subscript. Moreover, writing $A(x) = O_{\alpha,\beta} (B(x))$ or equivalently $A(x) \ll_{\alpha,\beta} B(x)$, indicates that the implicit constant depends on $\alpha$ and $\beta$.
	
	We use $q$ for a positive power of an odd prime, $\log_q$ for the logarithm in base $q$, $\Fq$ for the finite field with $q$ elements, $\Fq[t]$ for the polynomial ring over $\Fq$ in a variable $t$, and $\Fq(t)$ for its fraction field. Inside $\Fq[t]$, we consider the subsets $\calM$ of monic elements, $\calH$ of square-free elements, and $\calP$ of irreducible elements. A generic element of $\calP$ is denoted as $P$. For each integer $n \geq 1$, $\calM_n$, $\calH_n$ and $\calP_n$, respectively, denote the subsets of $\calM$, $\calH$ and $\calP$ of degree-$n$ elements and set $\calM_0 = \{1\}$. We let $\pi_q(n)$ be the cardinality of $\calP_n$.  We also write $\calM_{\leq n}$ (respectively $\calP_{\leq n}$) for the union of the $\calM_j$ (respectively of the $\calP_j$) with $1 \leq j \leq n$, and $\calM_{> n}$ as its complement in $\calM$. In particular, the cardinality of $\calP_{\leq n}$ is denoted $\Pi_q(n)$. Additionally, $\mathscr{P}(n)$ denotes the product of all $P \in \calP_{\leq n}$. The respective cardinalities of $\calM_n, \calH_n$ and $\calP_n$ are
	\begin{align*}
	&\# \calM_n = q^n,\numberthis\\
	&\# \calH_1 = q \text{ and } \#\calH_n = q^n - q^{n-1} \text{ if } n > 1,\numberthis\label{Cardinality_H_n}\\
	&\pi_q(n) = \frac{q^n}{n} + O\left(\frac{q^{\frac{n}{2}}}{n}\right),\numberthis\label{Prime_Number_Theorem}\\
	&\pi_q(n) \geq \frac{q^n}{n} - \frac{q^{\frac{n}{2}}}{n} - q^{\frac{n}{3}},\numberthis\label{Lower_Bound_Prime_Number_Theorem}
	\end{align*}
	see \cite{rosen_2002}*{Theorem 2.2 and its proof, and Proposition 2.3}). Furthermore, the notations $\sum_f$ and $\prod_P$ represent the limits $\lim_{n \to \infty} \sum_{f \in \calM_{\leq n}}$ and $\lim_{n \to \infty} \prod_{P \in \calP_{\leq n}}$.
	
\subsection{Affine Setting}

Each $f \in \calM$ of degree $d_f$ has norm $N(f) = q^{d_f}$. The zeta function of $\Fq[t]$ is defined as the series $\zeta(s,\Fq[t]) := \sum_f 1/N(f)^s$ in the complex variable $s$. This series converges absolutely for $\mathrm{Re}(s) > $1, as well as the product in \begin{equation}\label{Rational_and_Euler_Product_Zeta}  \frac{1}{1-q^{1-s}} = \zeta(s,\Fq[t]) = \prod_P \frac{1}{1 - N(P)^{-s}}. \end{equation}
We define the von Mangoldt function $\Lambda$ on $\calM$ by
$$
\Lambda(f) = \begin{cases}  \log(N(P)) &\text{if } f = P^m, \text{ for some integer } m \geq 1,\\
	0 &\text{otherwise}.
\end{cases}
$$

Taking the logarithmic derivatives in \eqref{Rational_and_Euler_Product_Zeta}, we find \begin{equation}\label{Log_Derivative_Zeta_von_Mangoldt}  \frac{\log(q)}{q^{s-1} - 1} =  -\frac{\zeta'(s,\Fq[t])}{\zeta(s,\Fq[t])} = \sum_f \frac{\Lambda(f)}{N(f)^s}.
\end{equation}

Next, letting $T:= q^{-s}$ in \eqref{Rational_and_Euler_Product_Zeta} gives $(1-qT)^{-1} = \prod_P (1-T^{d_P})^{-1}$. Taking the logarithmic derivatives and comparing their coefficients of the series expansions in $T$ yield \begin{equation}\label{Equation_Sum_von_Mangoldt_Same_degree}  \sum_{f \in \calM_n} \Lambda(f) = q^n, \text{ for all integers } n \geq 1.
\end{equation}
The following estimates are used in the text without always being explicitly mentioned.
\begin{lemma}\label{Lemma_Frequent_Computations}
	Let $y \geq 1$ be a real number. We have
	\begin{equation}\label{Equation_Finite_Sum_of_Norms}
		\sum_{N(f) \leq y} \frac{1}{N(f)} = \lfloor \log_q(y) \rfloor + 1, \sum_{N(f) \leq y} \frac{1}{N(f)^a} \ll_q 1 \text{ for an integer } a > 1
	\end{equation}
	and
	\begin{equation}\label{Equation_Average_Sum_von_Mangoldt_Same_degree}
		\sum_{N(f) \leq y} \frac{\Lambda(f)}{N(f)} = \lfloor \log_q(y) \rfloor.
	\end{equation}
\end{lemma}
\begin{proof}
	The assertions in \eqref{Equation_Finite_Sum_of_Norms} are direct computations and \eqref{Equation_Average_Sum_von_Mangoldt_Same_degree} follows from \eqref{Equation_Sum_von_Mangoldt_Same_degree}.
\end{proof}
Let $D \in \calM$ and $P \in \calP$. We set
\begin{equation}\label{Dirichlet_Character}
\chi_D(P) =
\begin{cases}
	-1&\text{if } D \pmod{P} \text{ is a non-square},\\
	0&\text{if } D \pmod{P} \text{ is zero},\\
	1&\text{if } D \pmod{P} \text{ is a non-zero square}.
\end{cases}
\end{equation}
We then define, for $f \in \calM$, the quantity $\chi_D(f) := \prod_{P|f} \chi_D(P)^{\ord_P(f)}$, where $\ord_P(f)$ is the $P$-adic valuation of $f$. This gives a real quadratic character modulo $D$ on $\calM$, denoted $\chi_D$.

The Dirichlet $L$-function of $\chi_D$ is defined, for $\re(s) > 1$, as
\begin{equation}\label{Definition_Dirichlet_L-function}
	L(s,\chi_D) := \sum_f \frac{\chi_D(f)}{N(f)^s} = \prod_{P \nmid D} \frac{1}{1 - \chi_D(P)N(P)^{-s}}.
\end{equation}
The series converges absolutely and so does the product. An analogous reasoning that yields \eqref{Log_Derivative_Zeta_von_Mangoldt}, shows that

\begin{equation}\label{Equation_Log_Derivative_Dirichlet_L-Function}
		\sum_P \frac{\chi_D(P)\Lambda(P)}{N(P)^s - \chi_D(P)}	= -\frac{L'(s,\chi_D)}{L(s,\chi_D)} = \sum_f \frac{\Lambda(f)\chi_D(f)}{N(f)^s}.
\end{equation}
	
From \cite{ihara_2008}*{Equation (6.8.20)} we have the following.

\begin{lemma}\label{Lemma_Ihara_Upper_Bound}
	For any integer $n \geq 1$ and $D \in \calH_n$, we have
	\begin{equation}
		\left| \frac{L'(s,\chi_D)}{L(s,\chi_D)}\right| \ll_q \begin{cases}
			
			\log^{2(1-\re(s))}(q^n) &\text{ if } \frac{1}{2} < \re(s) < 1,\\
			\log\left( \log(q^n)\right) & \text{ if } \re(s) = 1,\\
			1 & \text{ if } \re(s) > 1.
		\end{cases}
	\end{equation}
\end{lemma}

The next estimates are also useful.
\begin{lemma}\label{Lemma_Estimate_On_Sums_of_Quadratic_Characters}
	Let $f \in \calM \smallsetminus \{1\}$.
	\begin{enumerate}
		\item[(i)] If $f$ is a square in $\Fq[t]$, then
		$$
		\sum_{D \in \calH_n} \chi_D(f) = \#\calH_n \prod_{P|f} \left( \frac{N(P)}{N(P)+1}\right) + O_q\left( (\#\calH_n)^{1/2}\right).
		$$
		\item[(ii)] If $f$ is not a square in $\Fq[t]$, then for each $\varepsilon > 0$ we have
		$$
		\left|\sum_{D \in \calH_n} \chi_D(f)\right| \ll_{\varepsilon,q} (q^n)^{1/2} N(f)^\varepsilon.
		$$
	\end{enumerate}
\end{lemma}

To obtain these estimates, we modify and extend \cite{andrade_keating_2012}*{Proposition 5.2} and \cite{bui_florea_2018}*{Lemma 3.5} to all odd prime powers $q$.

\begin{proof}[Proof of Lemma \ref{Lemma_Estimate_On_Sums_of_Quadratic_Characters} (i)]
		
		The proof of \cite{andrade_keating_2012}*{Proposition 5.2} can be directly extended from $n \geq 1$ odd integer to any $n \geq 1$ integer.  and one gets
		$$
		\sum_{D \in \calH_n} \chi_D(f) =  \frac{q^n}{\zeta_q(2)} \prod_{P \mid f} \left(\frac{N(P)}{N(P)+1}\right) + O\left( \frac{q^n}{(q-1)q^{\lfloor n/2 \rfloor}} \prod_{P \mid f} \left(\frac{N(P)-1}{N(P) }\right)\right).
		$$
		Since $q^n/\zeta_q(2) = \#\calH_n$, $\prod_{P|f} (N(P)-1)/N(P) = O_q(1)$ and $q^n/((q-1)q^{\lfloor n/2 \rfloor} ) = O_q\left( \left(\# \calH_n\right)^{1/2} \right)$, the assertion is proven.
		\end{proof}

In \cite{bui_florea_2018}*{Lemma 3.5}, the authors proved Lemma \ref{Lemma_Estimate_On_Sums_of_Quadratic_Characters} (ii) assuming $n$ odd and, for simplicity, $q \equiv 1 \pmod{4}$. As suggested by \cite{bui_florea_2018}*{p.66}, it holds for all odd $q$. Their proof holds for all odd $q$ and all $n$ unless $q \equiv 3 \pmod{4}$ and both $n$ and $d_f$ are odd. We now consider the remaining case.

\begin{remark}\label{Remark_Lindelof}
In \cite{altug_tsimerman_2014}*{Theorem 3.3}, the authors prove that if $q$ is a prime number satisfying $q \equiv 1 \pmod{4}$ and $D \in \calH_n$ with $n \in \{2 g_D, 2 g_D + 1\}$, then the Lindel\"of bound  \begin{equation}\label{Lindelof}   \left| L\left( \frac{1}{2},\chi_D \right) \right| \leq e^{\frac{2g_D}{\log_q(g_D)} + 4 q^{\frac{1}{2}} g_D^{\frac{1}{2}}}  \end{equation}
	holds. Their proof of \eqref{Lindelof} holds for any prime power $q$. Combining \eqref{Lindelof} with the vertical periodicity of $L(s,\chi_D)$, we deduce that for all $s \in \C$ with $\re(s) = 1/2$ and all   $\varepsilon > 0$,  \begin{equation}\label{Lindelof_Epsilon}   \left| L(s, \chi_D) \right| \ll_{\varepsilon,q} (q^n)^{\varepsilon}.
	\end{equation}
\end{remark}

\begin{proof}[Proof of Lemma \ref{Lemma_Estimate_On_Sums_of_Quadratic_Characters} (ii) (when $q \equiv 3 \pmod{4}$ and both $n$ and $d_f$ are odd)]
Let $T := q^{-s}$ and $\calL(T,\chi_D) := L(s,\chi_D)$. By the quadratic reciprocity, multiplicativity of characters, and considering Euler factors of $L$-functions, we have  \begin{equation}\label{Identity_sum_characters}  \sum_{D \in \calH} \chi_D(f) T^{d_D} = \prod_{P \nmid f} \left( 1 + \chi_f(P)(-T)^{d_P} \right) = \frac{\calL(-T,\chi_f)}{\calL(T^2,\chi_f^2)} = \frac{\calL(-T,\chi_f) (1-q T^2)}{\prod_{P \mid f} \left( 1 - T^{2d_P} \right)}.
\end{equation}

Using Perron’s formula and \eqref{Identity_sum_characters}, we get  \begin{equation}\label{Perron_Formula}
	\sum_{D \in \calH_n} \chi_D(f)  = \frac{1}{2\pi i} \int_{|T| = q^{-\frac{1}{2}}} \frac{\calL(-T,\chi_f) (1-q T^2)}{T^{n+1}\prod_{P \mid f} \left( 1 - T^{2d_P} \right)}  dT.
\end{equation}

Now, write $f = f_1 f_2^2$ with $f_1 \in \calH$. Note that $\chi_f(P) = \chi_{f_1}(P)$ if $(P,f_1) = 1$, while if $P \mid f$ and $P \nmid f_1$ then $\chi_f(P) = 0$ and $\chi_{f_1}(P) \neq 0$. Hence,
$$
\calL(-T,\chi_f) = \calL(-T,\chi_{f_1}) \prod_{\substack{P \nmid f_1 \\ P \mid f_2}} \left( 1 - \chi_{f_1}(P) (-T)^{d_P} \right).
$$

Applying \eqref{Lindelof_Epsilon} to $\calL(-q^{-1/2},\chi_{f_1})$ in \eqref{Perron_Formula}, we have
\begin{equation}\label{Absolute_Value}
\left| \sum_{D \in \calH_n} \chi_D(f) \right| \ll_{\varepsilon,q} \tau(f) N(f_1)^{\frac{\varepsilon}{2}} \int_{|T| = q^{-\frac{1}{2}}} \left| \prod_{\substack{P \nmid f_1 \\ P \mid f_2}} \left( 1 - \chi_{f_1}(P) (-T)^{d_P}\right) \left( \frac{ 1 - q^2 T}{T^{n+1}}\right) \right| \mid dT  \mid,
\end{equation}
where we used that $\tau(f) = \sum_{d \mid f} d$ satisfies $\mid \prod_{P\mid f} 1/(1-T^{2d_P})\mid \leq  \tau(f)$. Now, for $\mid T \mid = q^{-\frac{1}{2}}$,
\begin{equation}\label{Bounding_Product}
\left| \prod_{\substack{P \nmid f_1 \\ P \mid f_2}} \left( 1 - \chi_{f_1}(P) (-T)^{d_P}\right) \left( \frac{ 1 - q^2 T}{T^{n+1}}\right) \right| \ll_q \prod_{\substack{P \nmid f_1 \\ P \mid f_2}} \left( 1 + q^{-\frac{d_P}{2}}\right) q^{ \frac{n}{2}} \ll_q 2^{\omega(f)} q^{\frac{n}{2}} \ll_q \tau(f) q^{\frac{n}{2}},
\end{equation}
where $\omega(f)$ is the number of distinct prime factors of $f$. As $\tau(f) \ll_\varepsilon N(f)^{\frac{\varepsilon}{2}}$, applying \eqref{Bounding_Product} into \eqref{Absolute_Value} concludes since
$$
\left| \sum_{D \in \calH_n} \chi_D(f) \right| \ll_{\varepsilon,q} q^{\frac{n}{2}} \tau(f) N(f_1)^{\frac{\varepsilon}{2}} \ll_{\varepsilon,q} \left(q^n\right)^{\frac{1}{2}} N(f)^{\varepsilon}. \qedhere
$$
\end{proof}

\subsection{Projective Setting}\label{Projective_Setting}
Fix an algebraic closure $\overline{\Fq(t)}$ of $\Fq(t)$. For any integer $n \geq 1$ and $D \in \calH_n$, let $\sqrt{D}$ be an element of $\overline{\Fq(t)}$ whose square is $D$. Let $K_D = \Fq(t)(\sqrt{D})$. The quadratic character $\psi_D$ of $\Gal(K_D/\Fq(t))$ is defined at each $P \in \calP$ by \begin{equation}\label{Artin_Character} \psi_D(P) = \begin{cases}
			-1 &\text{ if } P \text{ is inert in } K_D,\\
			0 &\text{ if } P \text{ ramifies in } K_D,\\
			1 &\text{ if } P \text{ splits in } K_D, \end{cases} \end{equation} and at the infinite place $\infty$ of $\Fq(t)$ as \begin{equation}\label{Artin_at_Infinity} \psi_D(\infty) = \begin{cases}
			0 &\text{ if } \infty \text{ ramifies in } K_D,\\
			1 &\text{ if } \infty \text{ splits in } K_D.
	\end{cases} \end{equation}
	The Artin $L$-function $L^{\mathrm{Art}}(s,\psi_D)$ attached to $\psi_D$ is defined as $$
	L^{\mathrm{Art}}(s,\psi_D) = \frac{1}{1-\psi_D(\infty)q^{-s}} \times \prod_P \frac{1}{1 - \psi_D(P)q^{-d_P s}}, $$
	which is holomorphic and non-vanishing on $\re(s) > 1$ (see \cite{rosen_2002}*{Proposition 9.15}).
	
	From the theory of characters (see \cite{rosen_2002}*{Proposition 14.9}) we have \begin{equation}\label{Identity_Zeta_Artin}  \zeta(s,K_D) = \zeta(s,\Fq(t))L^{\mathrm{Art}}(s,\psi_D).
	\end{equation}
	Also, \cite{rosen_2002}*{Propositions 14.6 and 17.7} gives \begin{equation}\label{Equation_LArtin_and_LDirichlet}
	L(s,\chi_D) =  \begin{cases}
			L^{\mathrm{Art}}(s,\psi_D) &\text{ if } \infty \text{ ramifies in } K_D, \text{ which happens when } n \text{ is odd,}\\   L^{\mathrm{Art}}(s,\psi_D)(1 - q^{-s}) &\text{ if } \infty \text{ splits in } K_D, \text{ which happens when } n \text{ is even}.
		\end{cases}
	\end{equation}
	
	Taking the logarithmic derivative of \eqref{Identity_Zeta_Artin} when $\re(s) > 1$ and using \eqref{Equation_LArtin_and_LDirichlet}, we obtain
\begin{equation}\label{Derivative_Identity_Zeta_Artin}
	\frac{\zeta'(s,K_D)}{\zeta(s,K_D)} = \begin{cases}
		\frac{\zeta'(s, \Fq(t))}{\zeta(s,\Fq(t))} + \frac{L'(s,\chi_D)}{L(s,\chi_D)} &\text{if } n \text{ is odd},\\
		\frac{\zeta'(s,\Fq(t))}{\zeta(s,\Fq(t))} + \frac{L'(s,\chi_D)}{L(s,\chi_D)} + (1-q^{-s})^{-1}&\text{if } n\text{ is even}.\\
	\end{cases}
\end{equation}
Hence, \eqref{Quadratic_Euler_Kronecker-0} follows by letting $s$ go to $1$.

\section{Computing Moments}

In this section, we compute integral moments of $-L'(1,\chi)/L(1,\chi)$. We start with a preparatory lemma that will used in Proposition \ref{Proposition_rthPower_LogDerivative} below.

	\begin{lemma}\label{Lemma_following_Apostol}
	For any real $0 < x < 1$ and positive integers $n$ and $r$, let
	$$
	\Phi(x,-r,n) := \sum_{k=0}^\infty (n+k)^r x^k.
	$$
	We have
	$$
	\mid \Phi(x, -r, n) \mid \leq \frac{(3nr)^r r!}{( 1-x)^{r+1}}.
	$$
\end{lemma}

\begin{proof}
	From \cite{Apostol_1951}*{p.164 and (3.2) and (3.7)}, $\beta_0(x) = 0$ (see \cite{Apostol_1951}*{p.165}) and $
	\frac{k}{r+1}\binom{r+1}{k} = \binom{r}{k-1}
	$, we have
	$$
	\Phi(x,-r,n) = -\sum_{k=1}^{r+1} \binom{r}{k-1}\left( \sum_{s=1}^{k-1} (-1)^s s! \frac{x^s}{(x-1)^{s+1}} \stirling{k-1}{s} \right) n^{r+1-k},
	$$
	where $\stirling{k-1}{s}$ are Stirling numbers of the second kind. Since $|x| < 1$ and $|x-1| < 1$, then
	$$
		\mid \Phi(x,-r,n) \mid \leq \frac{n^r r! }{|x-1|^{r+1}} \sum_{k=1}^{r+1} \binom{r}{k-1}\sum_{s=1}^{k-1} \stirling{k-1}{s}.
	$$
	Now $\stirling{k-1}{s} \leq \binom{k-1}{s} s^{k-1-s}$ (\cite{Rennie_Dobson_1969}*{Theorem 3}), and in the given ranges for $s$ and $k$ above we have $s^{k-1-s} < r^r$. Thus,
	$$
	\mid \Phi(x,-r,n) \mid \leq \frac{(nr)^rr! }{|x-1|^{r+1}} \sum_{k=1}^{r+1} \binom{r}{k-1}\sum_{s=1}^{k-1} \binom{k-1}{s},
	$$
	and as
	$$
	\sum_{k=1}^{r+1} \binom{r}{k-1}\sum_{s=1}^{k-1} \binom{k-1}{s} = \sum_{k=1}^{r+1} \binom{r}{k-1}(2^{k-1}-1) = 3^r-2^r,
	$$
	the result follows.
\end{proof}

We now  write powers of $-L'(s,\chi)/L(s,\chi)$ for $\re(s) > 1/2$ in terms of short Dirichlet polynomials. This is an improvement on its number field analogue \cite{lamzouri_2015}*{Proposition 2.3}.
\begin{proposition}\label{Proposition_rthPower_LogDerivative}
	Let $\chi$ be a non-trivial Dirichlet character. Let $c_0 > 0$ and $0 < \delta < c_0$ be real numbers. For any complex number $s$ with $\re(s) = c_0 + 1/2$, and any integer $r$ satisfying
	$$
	1 \leq r \leq \frac{\delta}{4}\frac{\log(q^n)}{\log(\log(q^n))},
	$$
	we have, for $n \geq \max\{1/(1-q^{-c_0}), 3r\}$,
	$$
	\left(-\frac{L'(s,\chi)}{L(s,\chi)}\right)^r = \sum_{f \in \calM_{\leq n-1}} \frac{\Lambda_r(f)\chi(f)}{N(f)^s} + O\left( (q^n)^{-c_0 + \delta} \right),
	$$
	where
	$$
	\Lambda_r(f) := \sum_{\substack{(f_1,\cdots,f_r) \in \calM^r\\ f_1 \cdots f_r = f} } \Lambda(f_1) \cdots \Lambda(f_r).
	$$
\end{proposition}
\begin{proof}
	Let $T := q^{-s}$. From \cite{rosen_2002}*{pp.40-41}, we have $L(s,\chi) = \calL(T,\chi) = \prod_{i=1}^{n-1}(1 - \alpha_i(\chi)T)$ and $|\alpha_i(\chi)| \in \{1, q^{1/2}\}$. For $\mid T \mid < q^{-1/2}$, we can write 
	$$
	T \frac{d\left( \log(\calL(T,\chi))\right)}{dT} =  \sum_{N=1}^\infty c_N(\chi) T^N,
	$$
	where $c_N(\chi) = -\sum_{i=1}^{n-1} \alpha_i(\chi)^N$. Then,
	\begin{equation}\label{Equation_rthpower}
	\left( T \frac{d\left( \log(\calL(T,\chi))\right)}{dT}  \right)^r = \sum_{N=r}^{n-1} \left( \sum_{\substack{(N_1,\cdots,N_r) \\ N_1 + \cdots + N_r = N}} \prod_{i=1}^r c_{N_i}(\chi)\right) T^N + \sum_{N=n}^{\infty} \left( \sum_{\substack{(N_1,\cdots,N_r) \\ N_1 + \cdots + N_r = N}} \prod_{i=1}^r c_{N_i}(\chi)\right) T^N.
	\end{equation}
	
	We start with the tail of \eqref{Equation_rthpower}. Since $\mid \alpha_i(\chi) \mid \leq q^{\frac{1}{2}}$ for $1 \leq i \leq n - 1$, then $\mid c_N(\chi) \mid \leq (n-1) q^{\frac{N}{2}}$ and therefore
	\begin{equation}\label{Equation_First_Bound}
		\left| \sum_{\substack{(N_1,\cdots,N_r) \\ N_1 + \cdots + N_r = N}} \prod_{i=1}^r c_{N_i}(\chi) \right| \leq \binom{N-1}{r-1}(n-1)^r q^{\frac{N}{2}} = \frac{(n-1)^r}{(r-1)!}(N - 1) \cdots (N -r +1) q^{\frac{N}{2}}.
	\end{equation}
	
	Let $x := 
	\mid q^{1/2}T \mid = q^{-c_0}$. Note that
	$$
	\frac{1}{(r-1)!} \sum_{N=n}^\infty (N - 1) \cdots (N -r +1)x^N \leq \frac{1}{(r-1)!} \sum_{N=n}^\infty N^{r-1} x^N.
	$$
	Then, by Lemma \ref{Lemma_following_Apostol}, we have
	\begin{equation}\label{Equation_Two_Conditions}
	\frac{1}{(r-1)!} \sum_{N=n}^\infty N^{r-1} x^N = \frac{1}{(r-1)!} x^n \Phi(x,-(r-1),n) \leq \frac{(3n(r-1))^{r-1}}{(1-x)^r}x^n.
	\end{equation}
	
	Hence, from \eqref{Equation_First_Bound} and \eqref{Equation_Two_Conditions}, we bound the tail of \eqref{Equation_rthpower} by
	\begin{equation}\label{Equation_Final_Tail}
		\sum_{N=n}^{\infty} \left| \sum_{\substack{(N_1,\cdots,N_r) \\ N_1 + \cdots + N_r = N}} \prod_{i=1}^r c_{N_i}(\chi)\right| |T|^N \leq (n-1)^r\left( \frac{(3n(r-1))^{r-1}}{(1-q^{-c_0})^r} \right)\left(q^{-c_0}\right)^n .
	\end{equation}

	We now focus on the head of \eqref{Equation_rthpower}. Since $c_N(\chi) = (1/\log(q))\sum_{f \in \calM_N}
	\Lambda(f) \chi(f)$ by \cite{rosen_2002}*{p.42, first displayed equation}, then
	\begin{equation}\label{Equation_Head_Computation}
	\sum_{N=r}^{n-1} \left( \sum_{\substack{(N_1,\cdots,N_r) \\ N_1 + \cdots + N_r = N}} \prod_{i=1}^r c_{N_i}(\chi)\right) T^N = \frac{1}{\log^r(q)}\sum_{f \in \calM_{\leq n-1}} \Lambda_r(f)\chi(f) T^{d_f}.
	\end{equation}
	
	Observe that
	$$
	T \frac{d\left( \log(\calL(T,\chi))\right)}{dT}  = -\frac{1}{\log(q)}\frac{L'(s,\chi)}{L(s,\chi)}.
	$$
	Hence, substituting \eqref{Equation_Final_Tail} and \eqref{Equation_Head_Computation} in  \eqref{Equation_rthpower} gives
	$$
	\left( -\frac{L'(s,\chi)}{L(s,\chi)}\right)^r = \sum_{f \in \calM_{\leq n-1}} \frac{\Lambda_r(f)\chi(f)}{N(f)^s} + O\left(  \frac{(n-1)^r(3n(r-1))^{r-1}\left( q^{-c_0}\right)^n\log^r(q)}{(1-q^{-c_0})^r}  \right).
	$$
	Now, for $n \geq \max\{1/(1-q^{-c_0}), 3r\}$, we have
	$$
	\left| \frac{(n-1)^r(3n(r-1))^{r-1}\left( q^{-c_0}\right)^n\log^r(q)}{(1-q^{-c_0})^r} \right| \ll_{c_0} \frac{\log^r(q) (3r)^r n^{3r}}{\left( q^{c_0} \right)^n} \ll_{c_0} \frac{\log^r(q) n^{4r}}{\left( q^{c_0} \right)^n}.
	$$
	Hence, for $n \geq \max\{1/(1-q^{-c_0}), 3r\}$, any $s \in \C$ with fixed $\re(s) = 1/2 + c_0 > 1/2$ satisfies
	$$
	\left( -\frac{L'(s,\chi)}{L(s,\chi)}\right)^r = \sum_{f \in \calM_{\leq n-1}} \frac{\Lambda_r(f)\chi(f)}{N(f)^s} + O\left( \frac{\log^{4r}(q^n)}{q^{n c_0}}\right).
	$$
	We conclude that
	$$
	\left( -\frac{L'(s,\chi)}{L(s,\chi)}\right)^r = \sum_{f \in \calM_{\leq n-1}} \frac{\Lambda_r(f)\chi(f)}{N(f)^s} + O\left( (q^n)^{-c_0+\delta}\right)
	$$
	because $\log^{4r}(q^n)/(q^{nc_0}) \leq (q^n)^{-c_0 + \delta}$ if and only if $r \leq (\delta/4)\log(q^n)/(\log(\log(q^n)))$.
	
\end{proof}
Next, we compute integral moments uniformly in the range of $r$ in Proposition \ref{Proposition_rthPower_LogDerivative} with respect to some parameters (For a number field analogue, see \cite{lamzouri_2015}*{ Theorem 2.1}.).

\begin{proposition}\label{Proposition_Average_Power_LprimeL}
	For a real number $0 < \delta < 1/2$ and integer $1 \leq r \leq \delta\log(q^n)/(4\log(\log(q^n)))$,
	\begin{equation}\label{Equation_Integral_Moments}
		\frac{1}{\# \calH_n} \sum_{D \in \calH_n} \left(-\frac{L'(1,\chi_D)}{L(1,\chi_D)}\right)^r = \sum_f \frac{\Lambda_r(f^2)}{N(f)^2} \prod_{P|f} \left( \frac{N(P)}{N(P)+1} \right) + O_{q,\delta}\left( (q^n)^{-\frac{1}{2}+\delta} \right).
	\end{equation}
\end{proposition}
\begin{proof}
	From Proposition \ref{Proposition_rthPower_LogDerivative} with $c_0 = 1/2$, we have
	\begin{equation}\label{Moments_Approximation}
		\frac{1}{\#\calH_n} \sum_{D \in \calH_n} \left(-\frac{L'(1,\chi_D)}{L(1,\chi_D)}\right)^r = \frac{1}{\#\calH_n}\sum_{D \in \calH_n} \sum_{f \in \calM_{\leq n-1}} \frac{\Lambda_r(f)\chi_D(f)}{N(f)} + O\left( (q^n)^{-\frac{1}{2} + \delta} \right).
	\end{equation}
	
	We now calculate the main term on the right-hand side of \eqref{Moments_Approximation}, beginning with the contribution of the square terms. By Lemma \ref{Lemma_Estimate_On_Sums_of_Quadratic_Characters} (i), we have
	\begin{equation}\label{All_Squares}
	\frac{1}{\#\calH_n}\sum_{D \in \calH_n} \sum_{\substack{f \in \calM_{\leq n-1},\\ f \text{ square}}} \frac{\Lambda_r(f)\chi_D(f)}{N(f)} = \sum_{f \in \calM_{\leq \frac{n-1}{2}}} \frac{\Lambda_r(f^2)}{N(f)^2} \prod_{P|f} \left(\frac{N(P)}{N(P)+1}\right) + O_q\left( \frac{1}{(\#\calH_n)^{1/2}} \right).
	\end{equation}
	
Because $2^r \log^r(x)/x^{1/2}$ decreases for real $x \geq e^{2r}$ and $\Lambda_r(f^2) \leq 2^r \log^r(N(f))$, then
	\begin{equation}\label{Large_Square}
	\sum_{f \in \calM_{> \frac{n-1}{2}}} \frac{\Lambda_r(f^2)}{N(f)^2} \prod_{P|f} \left( \frac{N(P)}{N(P)+1}\right) \ll_q \frac{\log^r(q^{n-1}) }{q^{\frac{n-1}{2}}} \ll_q \left( q^n \right)^{\frac{\delta}{4} - \frac{1}{2}}.
	\end{equation}

	Next, we analyze the effect of non-square terms on the right-hand side of  \eqref{Moments_Approximation}. As
	$\Lambda_r(f) \leq \log^r(q^n) \leq (q^n)^{\frac{\delta}{4}}$, then for $\varepsilon > 0$, we have, by \eqref{Equation_Finite_Sum_of_Norms} and Lemmas \ref{Lemma_Estimate_On_Sums_of_Quadratic_Characters} (ii) and \ref{Lemma_Frequent_Computations},
	\begin{equation}\label{All_Non_Squares}
	\frac{1}{\#\calH_n}\sum_{D \in \calH_n} \sum_{\substack{f \in \calM_{\leq n-1},\\ f \text{ non-square}}} \frac{\Lambda_r(f)\chi_D(f)}{N(f)} \ll_{\varepsilon,q} (q^n)^{\frac{\delta}{4}-\frac{1}{2}+\varepsilon}.
	\end{equation}
	Employing \eqref{Cardinality_H_n}, \eqref{All_Squares}, \eqref{Large_Square}, and \eqref{All_Non_Squares} into \eqref{Moments_Approximation} yields
	$$
	\frac{1}{\# \calH_n} \sum_{D \in \calH_n} \left( - \frac{L'(1,\chi_D)}{L(1,\chi_D)}\right)^r = 	\sum_f \frac{\Lambda_r(f^2)}{N(f)^2} \prod_{P|f} \left( \frac{N(P)}{N(P)+1} \right) + O_{\varepsilon,q}\left( (q^n)^{\frac{\delta}{4}-\frac{1}{2}+\varepsilon} + (q^n)^{\delta - \frac{1}{2}} \right).
	$$
	Choosing any $0 < \varepsilon < 3\delta/4$ yields the result.
\end{proof}

\begin{corollary}\label{Corollary_Estimate_Average_Bounded_by_One} For any fixed $r \geq 1$, we have
	$$
	\frac{1}{\#\calH_n} \sum_{D \in \calH_n} \left( -\frac{L'(1,\chi_D)}{L(1,\chi_D)} \right)^r \ll_{q,r} 1.
	$$
\end{corollary}

\begin{proof}
	It is enough to note that the main term of \eqref{Equation_Integral_Moments} is bounded from the above by
	$$
	\sum_f \frac{\Lambda_r(f^2)}{N(f)^2} \leq \sum_f \frac{ 2^r \log^r(N(f))}{N(f)^2} = 2^r\log^r(q) \sum_{m=0}^\infty \frac{m^r}{q^m} \ll_{q,r} 1.
	$$
\end{proof}
We next compute the Laplace transform of $-L'(1,\chi_D)/L(1,\chi_D)$. This will play a key part in Proposition \ref{Proposition_Probabilistic_Estimation_Exponential_Series} as a step towards establishing a limiting distribution problem.

\begin{proposition}\label{Proposition_Estimation_Exponential_Series}
	There is a real $m_{q,\delta} > 0$ such that for $s \in \C$ with $$
	|s| \leq m_{q,\delta} \frac{\log(q^n)}{\log^2(\log(q^n))},
	$$
	we have
	$$
	\frac{1}{\# \calH_n} \sum_{D \in \calH_n} \exp\left( s\left(- \frac{L'(1,\chi_D)}{L(1,\chi_D)}\right) \right) = \sum_{r=0}^N \frac{s^r}{r!} \sum_f \frac{\Lambda_r(f^2)}{N(f)^2} \prod_{P|f} \left( \frac{N(P)}{N(P)+1}\right) + O_{\delta,q}\left( (q^n)^{\frac{-\delta}{4\log(\log(q^n))}}\right),
	$$
	where $N := \left\lfloor \delta\log(q^n)/(4\log(\log(q^n))) \right\rfloor$ and $0 < \delta < 1/2$.
\end{proposition}

\begin{proof}
	Consider the Maclaurin expansion
	\begin{equation}\label{Equation_Deriving_Laplace_Transform}
		\exp\left(s\left(-\frac{L'(1,\chi_D)}{L(1,\chi_D)}\right)\right) = \sum_{r=0}^\infty \frac{s^r}{r!}\left( - \frac{L'(1,\chi_D)}{L(1,\chi_D)}\right)^r.
	\end{equation}
	
	From Stirling's approximation formula, we have $1/r! < (e/r)^r$ and by Lemma \ref{Lemma_Ihara_Upper_Bound}, there is a constant $C_q > 0$ such that $|L'(1,\chi_D)/L(1,\chi)| \leq C_q \log(\log(q^n)$. Combining these facts with the assumption $r > N$ yields
	$$
	E_1 := \sum_{r=N+1}^\infty \frac{s^r}{r!}\left( - \frac{L'(1,\chi_D)}{L(1,\chi_D)}\right)^r \leq \sum_{r=N+1}^\infty \left( \frac{C_q e |s| \log(\log(q^n)) }{N} \right)^r.
	$$
	Fix any real number $0 < m_{q,\delta} \leq \delta/(8 e^2 C_q)$. Since $1/N \leq 8\log(\log(q^n))/(\delta \log(q^n))$, then for any $s \in \C$ satisfying  $|s| \leq m_{q,\delta} \log(q^n)/\log^2(\log(q^n))$ we have
	\begin{equation}\label{Bound_on_E1}
	E_1 \leq \sum_{r=N+1}^\infty \left( \frac{e C_q m_{q,\delta}}{N} \frac{\log(q^n)}{\log(\log(q^n))}\right)^r \leq \sum_{r = N+1}^\infty \left( \frac{8 e C_q m_{q,\delta}}{\delta}\right)^r \leq \left(\frac{8 e C_q m_{q,\delta}}{\delta}\right)^N \leq e^{-N}.
	\end{equation}
	
	It follows from \eqref{Bound_on_E1} and Proposition \ref{Proposition_Average_Power_LprimeL} that
	\begin{align*}
		\frac{1}{\#\calH_n} \sum_{D \in \calH_n} \exp\left(s\left(-\frac{L'(1,\chi_D)}{L(1,\chi_D)}\right)\right)  &= \sum_{r=0}^N \frac{s^r}{r!} \sum_f \frac{\Lambda_r(f^2)}{N(f)^2} \prod_{P|f} \left( \frac{N(P)}{N(P)+1}\right)\\&+ O_{q,\delta}\left( \sum_{r=0}^N \frac{|s|^r}{r!} (q^n)^{-\frac{1}{2}+\delta} \right) + O\left( (q^n)^{\frac{-\delta}{4\log(\log(q^n))}} \right),
	\end{align*}
	which concludes the proof.
\end{proof}


\section{Probabilistic Model}
Let $(\pbx_P)_{P \in \calP}$ be a sequence of independent random variables valued in $\{-1,0,1\}$ with probabilities
$$
\prob\left( \pbx_P= a\right) =
\begin{cases}
	\frac{N(P)}{2(N(P)+1)} &\text{ if } a = -1,\\
	\frac{1}{N(P)+1} &\text{ if } a = 0,\\
	\frac{N(P)}{2(N(P)+1)} &\text{ if } a = 1.\\
\end{cases}
$$
We extend multiplicatively the above sequence to all of $\mathcal{M}$ by setting $\pbx_f := \prod_{P|f} \pbx_P^{\ord_P(f)}$ for each $f = \prod_{P|f} P^{\ord_P(f)}$ in $\mathcal{M}$, where $\ord_P(f)$ is the $P$-adic valuation of $f$. Observe that for $f \in \calM$ we have
\begin{equation}\label{Equation_Expectation_X_f}
	\Esp\left( \pbx_f \right) = \prod_{P|f} \Esp\left( \pbx_P^{\ord_P(f)} \right) =
	\begin{cases}
		\prod_{P|f} \left( \frac{N(P)}{N(P)+1} \right)&\text{ if } f \text{ is a square in } \Fq[t],\\
		0 &\text{otherwise}.
	\end{cases}
\end{equation}

\begin{lemma}\label{Equation_Equality_Almost_Surely}
The random series
	$$
		-\frac{L'(1,\pbx)}{L(1,\pbx)} :=\sum_f \frac{\Lambda(f)}{N(f)}\pbx_f  \text{ and } \sum_P \frac{\log(N(P))\pbx_P}{N(P)-\pbx_P}
	$$
	are almost surely convergent. Moreover, they are almost surely equal.
\end{lemma}
\begin{proof}
Over $\re(s) > 1$, the random series $\sum_f \frac{\Lambda(f)}{N(f)^s}\pbx_f$ and $\sum_P \frac{\log(N(P))}{N(P)^s - \pbx_P}$ converge absolutely, and we can write
\begin{equation}\label{Equation_Equality_Absolute_Convergence}
\sum_f \frac{\Lambda(f)}{N(f)^s}\pbx_f = \sum_{n=1}^\infty \sum_P \frac{\log(N(P))}{N(P)^{ns}}\pbx_P^n = \sum_P \frac{\log(N(P))\pbx}{N(P)^s-\pbx_P}.
\end{equation}

For $\re(s)>1/2$ and $P \in \calP$, define the random variables $\pby_{P,s} := \frac{\log(N(P))}{N(P)^s}\pbx_P$ and $\pbw_{P,s} := \frac{\log(N(P))\pbx_P}{N(P)^s - \pbx_P}$.

By Kolmogorov Theorem (see \cite{kowalski_2021}*{Theorem B.10.1}) the series $\sum_P \pby_{P,s}$ and $
\sum_P \pbw_{P,s}$ converge almost surely.

Next, for $\re(s)>1$, the random series identities
\begin{equation}
	\label{add-identity}
	\sum_P \frac{\log(N(P))}{N(P)^s - \pbx_P}
	=\sum_f \frac{\Lambda(f)}{N(f)^s}\pbx_f = \sum_P \frac{\log(N(P))}{N(P)^s}\pbx_P + O(1)
\end{equation}
hold. Since for a fixed $1/2 < \tau < 1$ the random series in \eqref{add-identity} are almost surely convergent, then by
\cite{kowalski_exponential_2021}*{Lemma A.4.1}, they define almost surely analytic functions on the half-plane $\re(s)>\tau$. Hence, \eqref{add-identity} extends to $\re(s) > \tau$, and thus the claimed assertion holds.
\end{proof}

\section{Comparing Laplace Transforms}
We now obtain a uniform bound on the expected value of the integral moments of $-\frac{L'(1,\pbx)}{L(1,\pbx)}$. It is the exact analogue of the number field version \cite{lamzouri_2015}*{Proposition 3.1}.
\begin{proposition}\label{Proposition_Upper_Bound_First_Moment_rth_Power_Logarithmic_Derivative}
There is a constant $c_q > 0$, such that for all integers $r \geq 8$ we have
\begin{equation}\label{Equation_Upper_Bound_Integral_Moment_Log_Derivative}
\Esp\left( \left| -\frac{L'(1,\pbx)}{L(1,\pbx)} \right|^r \right) \leq c_q^r \log^r(r).
\end{equation}
\end{proposition}
\begin{proof}
Let $n \geq 1$ be an integer to be chosen. Employing Minkowski's inequality yields
\begin{align*}
\Esp\left(  \left| -\frac{L'(1,\pbx)}{L(1,\pbx)} \right|^r \right)^{1/r} &\leq \Esp\left( \left| \sum_{f \in \calM_{\leq n}} \frac{\Lambda(f)}{N(f)} \pbx_f \right|^r \right)^{1/r} + \Esp\left( \left| \sum_{f \in \calM_{>n}} \frac{\Lambda(f)}{N(f)} \pbx_f \right|^r \right)^{1/r}\\
&\leq \sum_{f \in \calM_{\leq n}} \frac{\Lambda(f)}{N(f)} \Esp\left( \left|  \pbx_f \right|^r \right)^{1/r} + \Esp\left( \left| \sum_{f \in \calM_{> n}} \frac{\Lambda(f)}{N(f)} \pbx_f \right|^r \right)^{1/r}.
\end{align*}
By \eqref{Equation_Average_Sum_von_Mangoldt_Same_degree} we conclude that
\begin{equation}\label{Equation_Estimate}
	\Esp\left(  \left| -\frac{L'(1,\pbx)}{L(1,\pbx)} \right|^r \right)^{1/r} \ll_q \log(q^n) + \Esp\left( \left| \sum_{f \in \calM_{> n}} \frac{\Lambda(f)}{N(f)} \pbx_f \right|^r \right)^{1/r}.
\end{equation}

Applying the Cauchy-Schwarz inequality gives
\begin{equation}\label{Equation_Applying_Cauchy-Schwarz}
\Esp\left( \left| \sum_{f \in \calM_{> n}} \frac{\Lambda(f)}{N(f)} \pbx_f \right|^r \right) \leq \Esp\left(\left( \sum_{f \in \calM_{> n}} \frac{\Lambda(f)}{N(f)} \pbx_f \right)^2\right)^{1/2} \Esp\left(\left( \sum_{f \in \calM_{> n}} \frac{\Lambda(f)}{N(f)} \pbx_f \right)^{2(r-1)}\right)^{1/2}.
\end{equation}
Let
$$
\Lambda_{r,n}(f) := \sum_{\substack{f_1,\cdots,f_r \in \calM_{> n} \\ f_1 \cdots f_r  = f}} \Lambda(f_1) \cdots \Lambda(f_r).
$$
Using \eqref{Equation_Expectation_X_f} and $\Lambda_{r,n}(f) \leq \Lambda_r(f) \leq \log^r(f)$, we have
$$
\Esp\left(\left( \sum_{f \in \calM_{> n}} \frac{\Lambda(f)}{N(f)} \pbx_f \right)^{2m}\right) = \sum_{f \in \calM_{> mn}} \frac{(\Lambda_{2m, n}(f^2))^r}{N(f)^2}\Esp(\pbx_{f^2}) \leq \sum_{f \in \calM_{> mn}} \frac{2^{2m}\log^{2m}(N(f))}{N(f)^2}.
$$	
		Since $f(x)=2^{2m}\log^{2m}(x)/x^{1/2}$ is decreasing for $x \geq e^{4m}$, then 
		for $q^n \geq e^4$, we have
		$$
		\sum_{f \in \calM_{> mn}} \frac{2^{2m}\log^{2m}(N(f))}{N(f)^2} \leq \frac{(2m)^{2m}\log^{2m}(q^n)}{(q^n)^{m/2}} \sum_{f \in \calM_{> mn}} \frac{1}{N(f)^{3/2}} = \frac{ (2m)^{2m}\log^{2m}(q^n)}{(q^{1/2}-1)(q^n)^m}.
		$$
		Hence,
		$$
		\Esp\left(\left( \sum_{f \in \calM_{> n}} \frac{\Lambda(f)}{N(f)} \pbx_f \right)^{2m}\right) \ll_q \frac{ (2m)^{2m}\log^{2m}(q^n)}{(q^n)^m}.
		$$
		Employing this estimate in \eqref{Equation_Applying_Cauchy-Schwarz}, for $m=1$ and $m=r-1$, yields
\begin{equation}\label{Equation_Estimate_First_Moment_Euler_Kronecker}
\Esp\left( \left| \sum_{f \in \calM_{> n}} \frac{\Lambda(f)}{N(f)} \pbx_f \right|^r\right)^{1/r} \ll_q \frac{ r\log(q^n) }{(q^n)^{1/2}}.
\end{equation}

Choosing $q^n \geq r^2$ in  \eqref{Equation_Estimate_First_Moment_Euler_Kronecker} and substituting this estimate in  \eqref{Equation_Estimate} gives \eqref{Equation_Upper_Bound_Integral_Moment_Log_Derivative}.
	\end{proof}
	
We next compare the Laplace transforms of $-L'(1,\chi_D)/L(1,\chi_D)$ and $-\frac{L'(1,\pbx)}{L(1,\pbx)}$.
	
\begin{proposition}\label{Proposition_Probabilistic_Estimation_Exponential_Series}
		Let $0 < \delta < 1/2$ be a real number. There is a constant $M_{q,\delta} > 0$ such that for all $s \in \C$ with
		$$
		|s| \leq M_{q,\delta}\frac{\log(q^n)}{\log^2(\log(q^n))},
		$$
		we have
		$$
		\frac{1}{\# \calH_n} \sum_{D \in \calH_n} \exp\left( s\left(- \frac{L'(1,\chi_D)}{L(1,\chi_D)}\right) \right) =
		\Esp\left( \exp\left( s\left(-\frac{L'(1,\pbx)}{L(1,\pbx)} \right)\right) \right)
		+ O_{q,\delta}\left( (q^n)^{\frac{-\delta}{4\log(\log(q^n))}} \right),
		$$
		where $N := \left\lfloor\delta\log(q^n)/(4\log(\log(q^n))) \right\rfloor$.
	\end{proposition}
	
	\begin{proof}
		By Proposition \ref{Proposition_Upper_Bound_First_Moment_rth_Power_Logarithmic_Derivative}, the inequality $1/r! < (e/r)^r$, monotonicity of $f(r)=\log(r)/r$, and $r > N$, we have
		\begin{equation}\label{Equation_Bounding_Error_Estimate}
			\sum_{r=N+1}^\infty \frac{|s|^r}{r!}\Esp\left(\left| -\frac{L'(1,\pbx)}{L(1,\pbx)}\right|^r \right) \leq \sum_{r=N+1}^\infty\left( \frac{e c_q \log(N)|s|}{N} \right)^r.
\end{equation}

We now choose any real $0 < M_{q,\delta} \leq \delta/(16e^2c_q\log(\delta))$. Since $N \geq \delta \log(q^n)) /(8 \log(\log(q^n)))$, then for any $s \in \C$ satisfying $|s| \leq M_{q,\delta} \log(q^n)/(\log^2(\log(q^n))$ we have, for large $n$,
\begin{equation}\label{Equation_Approximation_Term}
\frac{e c_q \log(N) |s|}{N} \leq \frac{8 e c_q M_{q,\delta}}{\delta}\left( \frac{\log(\delta)}{\delta \log(\log(q^n))} + \frac{1}{\delta}\right) \leq e^{-1}.
\end{equation}

Now, employing \eqref{Equation_Approximation_Term} in \eqref{Equation_Bounding_Error_Estimate}, gives
\begin{equation}\label{add-eight}
\left| \Esp\left( \sum_{r=N+1}^\infty \frac{1}{r!} \cdot \left( s \cdot\left(-\frac{L'(1,\pbx)}{L(1,\pbx)}\right)\right)^r \right)\right| \leq e^{-N}.
\end{equation}

From the Maclaurin expansion of the exponential function, the linearity of $E$, \eqref{Equation_Expectation_X_f} and \eqref{add-eight} we have
		\begin{align*}
			\Esp\left( \exp\left( s\left(-\frac{L'(1,\pbx)}{L(1,\pbx)} \right)\right) \right)
			&= 
			\sum_{r=0}^N \frac{s^r}{r!} \sum_f \frac{\Lambda_r(f^2)}{N(f)^2} \prod_{P|f} \left( \frac{N(P)}{N(P)+1}\right)  + O\left( (q^n)^{\frac{-\delta}{4\log(\log(q^n))}} \right).
		\end{align*}
		
		This identity together with Proposition \ref{Proposition_Estimation_Exponential_Series} imply the result.
	\end{proof}
\section{Exponential Decay}
The characteristic function $\Phi_{\rand}$ of $-\frac{L'(1,\pbx)}{L(1,\pbx)}$ is given for $u \in \R$ by
\begin{equation}\label{Characteristic_Function_rand}
\Phi_{\rand}(u) := \Esp\left( \exp\left( iu \left( -\frac{L'(1,\pbx)}{L(1,\pbx)}\right) \right) \right).
\end{equation}

Our next result implies that  $\Phi_{\rand}(u)$ decays exponentially as $|u|$ goes to infinity. Hence, the distribution function $F_\rand$ of $-\frac{L'(1,\pbx)}{L(1,\pbx)}$ has a density function $M_{\rand}$, given by the Fourier inversion formula
$$
M_{\rand}(y) = \frac{1}{2\pi} \int_{-\infty}^\infty \exp(-iyu)\Phi_{\rand}(u)dy,
$$
that is smooth (see \cite{folland_1999}*{Theorem 8.22 (d)}) and
\begin{equation}\label{density}
F_\rand(x) = \int_{-\infty}^x M_{\rand}(y) dy
\end{equation}
with $\sup_{x \in \R} F_\rand'(x)=\sup_{x\in \R} M_{\rand}(x) $ finite, where $F'_\rand$ is the derivative of $F_\rand$.
\begin{proposition}\label{Proposition_Exponential_Decay}
	For $0 < \varepsilon < 1$, there is a real $C_{q,\varepsilon} > 0$ such that for large $|u|$ we have
	$$
	\left| \Phi_{\rand}(u) \right| \leq \exp\left( - C_{q,\varepsilon} |u|^{1-\varepsilon} \right).
	$$
\end{proposition}

\begin{proof}
	Since the sequence $( \sum_{P \in \calP_{\leq m}} \log(N(P))\pbx_P/(N(P) - \pbx_P))_{m \geq 1}$ converges almost surely, thus converges weakly \cite{billingsley_2012}*{Theorem 25.2}, to $-\frac{L'(1,\pbx)}{L(1,\pbx)}$. As $\exp(iu)$ is continuous and bounded, then \cite{billingsley_2012}*{Theorem 25.8} and the independence of $\{\log(N(P))\pbx_P/(N(P) - \pbx_P) : P \in \calP\}$ give
	$$
		\Phi_{\rand}(u) = \lim_{m \to \infty} \Esp\left( \exp\left(iu \sum_{P \in \calP_{\leq m}} \frac{\log(N(P))\pbx_P}{N(P) - \pbx_P} \right) \right) = \prod_P \Esp\left( \exp \left(iu \frac{\log(N(P))\pbx_P}{N(P) - \pbx_P} \right) \right).
		$$

	Each factor $M_P(u)$ of the above product equals
	\begin{equation}\label{Euler_Factor_MP}
	M_P(u) = \frac{1}{N(P) + 1} + \frac{N(P)}{2(N(P)+1)}\left( \exp\left(-iu\frac{\log(N(P))}{N(P)+1}\right) + \exp\left( iu \frac{\log(N(P))}{N(P)-1} \right)\right).
	\end{equation}
	Using $(N(P) \pm 1)^{-1} - N(P)^{-1} \ll N(P)^{-2}$ and $| \exp(ib) - \exp(ia) | \leq | b - a |$ in \eqref{Euler_Factor_MP} yields
	\begin{equation}\label{MP_Using_Cosine}
	M_P(u) = \frac{1}{N(P)+1} + \frac{N(P)}{N(P)+1}\cos\left( u \frac{\log(N(P))}{N(P)} \right) + O\left( \frac{|u|\log(N(P))}{N(P)^2}\right).
	\end{equation}
From \eqref{Euler_Factor_MP}, $| M_P(u) | \leq 1$ for all $P \in \calP$.
Let $\theta > 0$ be such that if $N(P) > \theta$ then $|u|\log(N(P))/N(P) < 1$. By taking the Maclaurin expansion of the cosine function and using $\log(1-x) \leq -x$ for $x < 1$, we get from \eqref{MP_Using_Cosine}:
	\begin{align}\label{Upper_Bound_Euler_Factor}
		\left| \Phi_\rand(u) \right| \leq
	\prod_{P \in \calP_{> \theta}} \left|M_P(u) \right| \leq \exp\biggl( &-|u|^2 \sum_{P \in \calP_{> \theta}} \frac{\log^2(N(P))}{2N(P)(N(P)+1)}\\
	&+ O\left(|u|^4 \sum_{P \in \calP_{> \theta}}\nonumber \frac{\log^4(N(P))}{N(P)^4}\ \right)
	+ O\left.\left( |u| \sum_{P \in \calP_{> \theta}}\frac{\log(N(P))}{N(P)^2}\right) \right).\nonumber
	\end{align}
	
	Next, we estimate the three summands in \eqref{Upper_Bound_Euler_Factor}. Let $N_\theta$ be the smallest integer $N > 0$ satisfying $q^N > \theta$. Using \eqref{Prime_Number_Theorem}, we have
	\begin{equation}\label{Equation_First_Summand}
		\sum_{P \in \calP_{> \theta}} \frac{\log^2(N(P))}{2N(P)(N(P)+1)} = \frac{\log^2(q)}{2}\left( \sum_{m=N_\theta}^\infty \left( \frac{m}{q^m+1} + O\left( \frac{m}{(q^m)^{1/2}(q^m  + 1)}\right) \right) \right).
	\end{equation}
	
	Observe that there is a constant $\alpha_q > 0$ such that
	\begin{equation}\label{Constant_First_Summand}
	\alpha_q \sum_{m = N_\theta}^\infty \frac{m}{q^m} \leq \sum_{m = N_\theta}^\infty \frac{m}{q^m + 1} \leq \sum_{m = N_\theta}^\infty \frac{m}{q^m}.
	\end{equation}
	Combining \eqref{Equation_First_Summand} and \eqref{Constant_First_Summand} with the fact that $q^{N_\theta} > \theta$ implies that
\begin{equation}\label{First_Estimate_Equation}
\sum_{P \in \calP_{> \theta}} \frac{\log^2(N(P))}{2N(P)(N(P)+1)} \asymp_q   \frac{\log(\theta)}{\theta}.
\end{equation}
Using analogous arguments, we find
\begin{equation}\label{Second_Estimate_Equation}
\sum_{P \in \calP_{> \theta}}  \frac{\log^4(N(P))}{N(P)^4}  = \log^4(q) \sum_{m = N_\theta}^\infty \frac{m^4}{q^{4m}} \pi_q(m) = O_q\left(  \frac{\log^3(\theta)}{\theta^3} \right)
\end{equation}
and
\begin{equation}\label{Third_Estimate_Equation}
\sum_{P \in \calP_{> \theta}} \frac{\log(N(P))}{N(P)^2} = \log(q) \sum_{m = N_\theta}^\infty \frac{m}{q^{2m}} \pi_q(m) = O_q\left( \frac{1}{\theta}\right).
\end{equation}

For $0 < \varepsilon < 1$ and $\theta = |u|^{1+\varepsilon}$, there are constants $A_q, B_{q,\varepsilon}, C_q > 0$ such that
		\begin{align*}
			\left| \Esp\left( \exp\left(-iu\frac{L'(1,\pbx)}{L(1,\pbx)}\right)\right)\right| \leq \exp\left( -A_q \frac{|u|^2}{\theta} + B_{q,\varepsilon}  \frac{|u|^4}{\theta^{3-\varepsilon}} + C_q \frac{|u|}{\theta}\right) \leq \exp\left( - C_{q,\varepsilon} |u|^{1-\varepsilon} \right)
		\end{align*}
		for a suitable constant $C_{q, \varepsilon}>0$ for large $|u|$.
	\end{proof}
	
\section{Discrepancy and Proof of Theorem \ref{Theorem_Application_Small_Absolute_Value}}
For each integer $n \geq 1$, define an arithmetic function $f_n : \calH_n \to \R$ by setting $f_n(D) := -L'(1,\chi_D)/L(1,\chi_D)$ for each $D \in \calH_n$. To $f_n$, we associate a function $F_n : \R \to [0,1]$ by setting for each $x \in \R$,
	$$
	F_n(x) := \frac{\#\{ D \in \calH_n : f_n(D) \leq x \}}{\# \calH_n} ,
	$$
	which is a distribution function. Its characteristic function $\varphi_n$ is defined for all $u \in \R$ by
	$$
	\varphi_n(u) := \int_{-\infty}^{\infty} \exp(iux) dF_n(x).
	$$
	For each $D \in \calH_n$ and $x \in \R$, we set $\psi_{f_n(D)}(x)$ to be $1$ if $f_n(D) \leq x$ and $0$ otherwise. We have
	$$
	\varphi_n(u) = \frac{1}{\# \calH_n} \sum_{D \in \calH_n} \int_{-\infty}^\infty \exp(iux)d\left( \psi_{f_n(D)}(x)\right) = \frac{1}{\# \calH_n} \sum_{D \in \calH_n} \exp\left( iu \left( - \frac{L'(1,\chi_D)}{L(1,\chi_D)}\right)\right).
	$$
	
Recall that the random variable $-\frac{L'(1,\pbx)}{L(1,\pbx)}$ has distribution function $F_\rand$ and characteristic function $\Phi_\rand$. We now obtain an upper bound on the discrepancy between the distribution of $-L'(1,\chi_D)/L(1,\chi_D)$ and of its random model.

\begin{proposition}\label{Proposition_Discrepancy}
We have
$$
\sup_{x \in \R} \left| F_n(x) - F_\rand(x) \right| \ll_q \frac{\log^2(\log(q^n))}{\log(q^n)}.
$$
\end{proposition}
	
\begin{proof} For any real $R > 0$, Berry-Esseen inequality \cite{loeve_1977}*{p.297, A. Basic Inequality} gives
\begin{equation}\label{Berry_Esseen}
\sup_{x \in \R} \left| F_n(x) - F_\rand(x) \right| \leq \frac{1}{\pi} \int_{-R}^R \left| \frac{\varphi_n(u) - \Phi_\rand(u)}{u}\right|du + \frac{24}{\pi}\frac{\sup_{x \in \R} F_\rand'(x)}{R}.
\end{equation}
Fix $0 < \delta_0 < 1/2$ and let $M_{q,\delta_0} > 0$ be the constant from Proposition \ref{Proposition_Probabilistic_Estimation_Exponential_Series}. For $R = M_{q,\delta_0}\log(q^n)/ \log^2(\log(q^n))$, $R_0 = 1/\log(q^n)$ and $n$ large enough, we have $\log(R/R_0) > 0$. By Proposition \ref{Proposition_Probabilistic_Estimation_Exponential_Series}, we have
		\begin{align*}
			\int_{[-R,R] \smallsetminus [-R_0,R_0]} \left| \frac{\varphi_n(u) - \Phi_\rand(u)}{u}\right|du &\ll  \log\left( \frac{R}{R_0} \right) \sup\limits_{u \in [-R,R]} \left| \varphi_n(u) - \Phi_\rand(u) \right|\\
			&\ll_q  (q^n)^{-\frac{\delta}{4\log(\log(q^n))}}\log(\log(q^n)).\numberthis\label{Outer_Rectangle}
		\end{align*}		

Next, since $|e^{i\theta} - 1| \ll | \theta |$ for all $\theta \in \R$, we have
		\begin{align*}
			\varphi_n(u) - \Phi_\rand(u) &= \int_{-\infty}^{+ \infty} ( e^{ixu} -1)d F_n(x) - \int_{-\infty}^{+ \infty} ( e^{ixu} -1)d F_\rand(x)\\
			&\ll |u|\left( \left( \int_{-\infty}^{+\infty} x^2 d F_n(x) \right)^{1/2} + \left( \int_{-\infty}^{+\infty} x^2 d F_\rand(x) \right)^{1/2}\right)\\
			&\ll |u|,
		\end{align*}
		where the last estimate follows from the bound of Corollary \ref{Corollary_Estimate_Average_Bounded_by_One} with $r = 2$. Therefore,
		\begin{equation}\label{Estimate_Inner_Square}
		\int_{-R_0}^{R_0} \left| \frac{\varphi_n(u)-\Phi_\rand(u)}{u} \right| du \ll \frac{1}{\log(q^n)}.
		\end{equation}
		
		Since $\sup_{x \in \R} F_\rand'(x)$ is bounded, then we obtain from \eqref{Berry_Esseen}, \eqref{Outer_Rectangle} and \eqref{Estimate_Inner_Square} that
		$$
		\sup_{x \in \R} \left| F_n(x) - F_\rand(x) \right|\ll_q \frac{\log^2(\log(q^n))}{\log(q^n)}. \qedhere
		$$
	\end{proof}

For the rest of the section, we fix an ordering $(P_j)_{j \geq 1}$ on $\calP$ so that $d_{P_j} \leq d_{P_{j+1}}$ for all $j \geq 1$. Let $\calP_E$ and $\calP_O$ be respectively the set of monic irreducible polynomials in $\Fq[t]$ of even and odd degrees. We consider them ordered with the ordering coming from $\calP$. We next prove a technical lemma towards showing that $M_\rand$ is positive.

 \begin{lemma}\label{series-density}
			For $P\in \calP_O$, define a function $h_P$ by $h_P(x)=(\log{N(P)}x)/(N(P)-x)$ with $x\in \mathbb{R}\smallsetminus\{q^{d_P}\}$. For fixed reals $\varepsilon>0$ and $\alpha$, there is a $\{-1, 0, 1\}$-valued sequence  $(x_P)_{P \in \calP_O}$ such that \begin{equation}\label{Density_Equation}
				\left| \sum_{P \in \calP_O} h_P(x_P)-\alpha\right|<\varepsilon.
			\end{equation}
		\end{lemma}
		
		\begin{proof} If $|\alpha| < \varepsilon$, the zero sequence satisfies \eqref{Density_Equation}. Now suppose $|\alpha| \geq \varepsilon$ and write $\alpha = \omega \beta$ with $\omega \in \{-1,1\}$ and $\beta > 0$. Let $(P_i)_{i \geq 1}$ be the elements of $\calP_O$ with the induced order from $\calP$. Since $\omega h_{P_i}(\omega)>0$ for all $i \geq 1$ and the sequence $(\omega h_{P_i}(\omega))_{i \geq 1}$ converges to $0$, there is an integer $m \geq 1$ such that for $i\geq m$, we have $0<\omega h_{P_i}(\omega)<\min\{\varepsilon/2, \beta-\varepsilon/2\} $. Now observe that $\sum_{i=1}^{\infty} \omega h_{P_i}(\omega)$ diverges to $\infty$. So, there is an integer $n>m$ such that  $\sum_{i=m}^{n-1} \omega h_{P_i}(\omega) \leq \beta-\varepsilon/2$ and  $\sum_{i=m}^n \omega h_{P_i}(\omega) > \beta-\varepsilon/2$. Since $n>m$, we have $\omega h_{P_n}(\omega)<\varepsilon/2$. Hence,
		$$
	  \left|\sum_{i=m}^n h_{P_i}(\omega)- \alpha\right| = \left| \sum_{i=m}^n \omega h_{P_i}(\omega) - \beta \right| <\varepsilon.
	$$
To satisfy \eqref{Density_Equation}, we can then choose $x_{P_i} = \omega $ for $m\leq i\leq n$ and $x_{P_i}=0$ otherwise.
\end{proof}

\begin{proposition}\label{Proposition_Positive_Density}
		The density function $M_\rand$ in \eqref{density} is positive.
\end{proposition}
	
	\begin{proof}
		For $\nu \in \{0,1\}$ and integer $j \geq 1$, let
		$$
		a_{P_j, \nu}= \frac{(-1)^\nu \log{N(P_j)}}{N(P_j)-(-1)^\nu}.$$
		Let
		$$
		\delta(x) = \begin{cases}
			0 &\text{if } x < 0,\\
			1 &\text{if } x \geq 0.
		\end{cases}
		$$
		Now, for all $j \geq 1$, set $\pbw_{P_j} := \pbw_{P_j,1}$ as in Lemma \ref{Equation_Equality_Almost_Surely}, and
		$$
		F_{\pbw_{P_j}}(x):=\frac{1}{N(P_j) + 1} \delta(x)+ \frac{N(P_j)}{2(N(P_j)+1)}\left( \delta(x-a_{P_j,0}) + \delta(x-a_{P_j,1})\right).
		$$
		
		 We can write \eqref{Characteristic_Function_rand} as
		$$\Phi_\rand(u) = \lim_{n \to \infty} \prod_{j\leq n} \Esp\left( \exp(iu \pbw_{P_j}) \right).$$
		Since $\Phi_\rand(0)=1$ and  $\Phi_\rand$ is continuous at $0$, there is a real number $\eta > 0$ such that for each $u \in [-\eta,\eta]$ we have $\Phi_\rand(u) \neq 0$ and so
		$$
		\lim_{m,n \to \infty} \prod_{m < j \leq n} \Esp\left( \exp(iu \pbw_{P_j}) \right) = 1.
		$$
		Hence, by \cite{tenenbaum_1995}*{Theorem 2.7}, the sequence $( F_{\pbw_{P_1}}*\cdots*F_{\pbw_{P_n}})_{n \geq 1}$  of convolution products converges weakly to $F_\rand$. Similarly, for  $(P_j)_{j \geq 1}$ in $\calP_E$, respectively in $\calP_O$, the sequence $( F_{\pbw_{P_1}}*\cdots*F_{\pbw_{P_n}})_{n \geq 1}$ of convolution products converges weakly to a distribution function $F_E$, respectively to a distribution function $F_O$. Then we get  $F_\rand = F_E * F_O$. As in the proof of Proposition \ref{Proposition_Exponential_Decay}, $F_E$ and $F_O$ have smooth density functions $M_E$ and $M_O$. Moreover,
\begin{equation}\label{convol}
	M_\rand(x) = \int_{-\infty}^\infty M_O(x-u) M_E(u) du.
\end{equation}

Now, by the Kolmogorov Theorem, the random series
\begin{equation}\label{Odd_Log_Derivative_Proba}
	\sum_{P \in \calP_O} \frac{\log(N(P))\pbx_P}{N(P)-\pbx_P}
\end{equation}
converges almost surely. By Lemma \ref{series-density}, the set 
$$
\left\{ \sum_{P \in \calP_O} \frac{\log(N(P))x_P}{N(P)-x_P} < \infty : x_P \in \{-1,0,1\} \right\}
$$
is dense in $\R$. Hence, from \cite{kowalski_2021}*{Proposition B.10.8}, the support of \eqref{Odd_Log_Derivative_Proba} is $\R$. Therefore, for $a, b \in \R$ with $a < b$, we have $F_O(b) - F_O(a) = \int_a^b M_O(x) dx > 0$. In particular, $M_O$ is not identically zero on any interval $]a,b[$. Since $\int_{-\infty}^\infty M_E(u) du = 1$ and $M_E$ is continuous, then $M_E > 0$ on one of these intervals $]a,b[$. From \eqref{convol} and the continuity of $M_O$, there is a non-empty open sub-interval $]a_1,b_1[ \subset ]a,b[$ on which $M_O(x-u) > 0$. Therefore,
$$
M_\rand(x) \geq \int_{a_1}^{b_1} M_E(x-u) M_O(u) du > 0.
$$
Since $x \in \R$ was arbitrary, then the function $M_\rand$ is positive on $\R$.
\end{proof}

We conclude this section by proving Theorem \ref{Theorem_Application_Small_Absolute_Value}.
\begin{proof}
Let $\{ \varepsilon_n \}_{n \geq 1}$ be a sequence of real numbers converging to $0$. By Proposition \ref{Proposition_Discrepancy},
		$$
		\frac{\# \left\{ D \in \calH_n : \left| \frac{L'(1,\chi_D}{L(1,\chi_D)} \right| \leq \varepsilon_n \right\} }{\# \calH_n} = F_n(\varepsilon_n) - F_n(-\varepsilon_n) = F_\rand(\varepsilon_n) - F_\rand(-\varepsilon_n) + O_q\left( \frac{\log^2(\log(q^n))}{\log(q^n)}\right).
		$$
		
		Since $F_\rand$ has a continuous density function $M_\rand$, which is positive by Proposition \ref{Proposition_Positive_Density},
		$$
		F_\rand(\varepsilon_n) - F_\rand(-\varepsilon_n) = \int_{-\varepsilon_n}^{\varepsilon_n} M_\rand(y) dy \gg_q \varepsilon_n.
		$$
		Thus, by choosing $\varepsilon_n = C_q\log^2(\log(q^n)))/\log(q^n)$ for a large enough  constant $C_q > 0$, we obtain
		$$
		\frac{\# \left\{ D \in \calH_n : \left| \frac{L'(1,\chi_D}{L(1,\chi_D)} \right| \leq \varepsilon_n \right\} }{\# \calH_n} \gg_q \frac{\log^2(\log(q^n))}{\log(q^n)},
		$$
		which concludes the proof.
	\end{proof}
	
\section{Proof of Theorem \ref{Theorem_Omega_Result}}

The next two lemmas will be used to prove the main result of this section.
\begin{lemma}\label{Lemma_Sum_allr}
	For any integer $m \geq 1$, we have
	$$
	\sum_{\substack{P \in \calP_{\leq m} \\ r \geq 1~\mathrm{odd}}} \frac{\log(q^{d_P})}{(q^{d_P})^r} = m\log(q) + O_q(1).
	$$
\end{lemma}
\begin{proof}
	Observe that
	\begin{equation}\label{Equation_Decomposition}
	\sum_{\substack{P \in \calP_{\leq m} \\ r \geq 1 \text{ odd}}} \frac{\log(q^{d_P})}{(q^{d_P})^r} = \sum_{P \in \calP_{\leq m}} \frac{\log(q^{d_P})}{q^{d_P}} + \sum_{\substack{P \in \calP_{\leq m} \\ r \geq 3 \text{ odd}}} \frac{\log(q^{d_P})}{(q^{d_P})^r}.
	\end{equation}
	The result follows from \eqref{Prime_Number_Theorem}, writing
	\begin{equation}\label{Equation_Sum_risone}
		\sum_{P \in \calP_{\leq m}} \frac{\log(q^{d_P})}{q^{d_P}} = \log(q)\sum_{j=1}^m \frac{j \pi_q(j)}{q^{j}} = m\log(q) + O_q(1),
	\end{equation}
	and bounding the second summand of \eqref{Equation_Decomposition} by $\log(q)\sum_{r=3}^\infty \sum_{j=1}^m j \pi_q(j)/q^{rj}$.
\end{proof}
\begin{lemma}\label{Lemma_Lower_Bound}
For any integer $m \geq 1$, we have
$$
\sum_{P \in \calP_{\leq m}} \frac{\log(q^{d_P})}{q^{d_P}} > (m -2.61)\log(q).
$$	
\end{lemma}
\begin{proof}
	Using \eqref{Lower_Bound_Prime_Number_Theorem}, we have
	\begin{align*}
		\sum_{P \in \calP_{\leq m}} \frac{\log(q^{d_P})}{q^{d_P}}
		\geq \left(m - \frac{1}{q^{1/2}-1} - \frac{q^{1/3}}{(q^{2/3}-1)^2}\right)\log(q).
	\end{align*}
	We conclude because the function $f(x) = x^{1/3}/(x^{2/3}-1)^2 + 1/(x^{1/2}-1)$ is strictly decreasing for real $x > 1$ and maximized at $3$, in terms of odd prime powers, and $f(3) = 2.60233 \ldots$.
\end{proof}

Let $m, n \geq 1$ be integers. For each $P \in \calP_{\leq m}$, choose a value $\delta_P \in \{-1,1\}$ and let
$$
S(n, m, (\delta)_m) := \left\{ Q \in \calP_n \mid \text{ if } P \in \calP_{\leq m}, \text{ then } \left( \frac{P}{Q} \right) =  \delta_P \right\}.
$$
Recall that $\Pi_q(m) = \sum_{j=1}^m \pi_q(j)$ and $\mathscr{P}(m) = \prod_{P \in \calP_{\leq m}} P$.
\begin{proposition}\label{Proposition_Intermediate_Omega_Step}
	For integer $m \geq 1$ and large enough integer $n \geq  q^{m^{1/2}}$, we have
	$$
	\sum_{Q \in S(n, m, (\delta)_m)} \left(-\frac{L'(1,\chi_Q)}{L(1,\chi_Q)}\right) = \frac{\pi_q(n)}{2^{\Pi_q(m)}}
	\sum_{\substack{P \in \calP_{\leq m} \\ r \geq 1 \text{ odd}}} (-1)^{\frac{(q-1)n d_P}{2}} \delta_P  \frac{\log(q^{d_P})}{(q^{d_P})^r} + O_q\left( n^2 q^{\frac{n}{2} + 2m} \right).
	$$
\end{proposition}
\begin{proof}
	For each $f \in \calM$, let $\delta_f := \prod_{P \mid f} \delta_P$. For a given $Q \in \calP_n$, one has
	\begin{equation}\label{Equation_sum_Snepsilon}
		\sum_{f \mid \mathscr{P}(m)} \delta_f \chi_f(Q) = \prod_{P \in \calP_{\leq m}} \left( 1 + \delta_P \chi_P(Q) \right) = \begin{cases}
			2^{\Pi_q(m)} &\text{ if } Q \in S(n, m, (\delta)_m),\\
			0 &\text{ otherwise}.
		\end{cases}
	\end{equation}
	As shown in \cite{lumley_2019}*{Lemma 6.1}, for large enough $n \geq q^{m^{1/2}}$, we have
	\begin{equation}\label{Equation_Cardinality_Snm}
		\# S(n, m, (\delta)_m) = \frac{q^n}{2^{\Pi_q(m)}n} + O_q\left( q^{\frac{n}{2} + m}\right).
	\end{equation}
	
	For any real number $y > 1$ and any $Q \in \calP_n$ it follows from \cite{ihara_2008}*{Equation (6.8.4)} that 
	\begin{equation}\label{Ihara_Approx}
	-\frac{L'(1,\chi_Q)}{L(1,\chi_Q)} = \sum_{f \in \calM_{\leq \lfloor \log_q(y) \rfloor}} \frac{\chi_Q(f)\Lambda(f)}{N(f)} + O_q\left( \frac{\log(q^n)}{y^{\frac{1}{2}}}\right).
	\end{equation}
	
	Now, let $y := q^{2 n d_{\mathscr{P}(m)}}$ so that $\rho(n,m) := \log_q(y) = 2 nd_{\mathscr{P}(m)}$. From \eqref{Equation_Cardinality_Snm} we have
	\begin{equation}\label{Bound_Error_Ihara}
		\# S(n, m, (\delta)_m) \frac{\log(q^n)}{q^{n d_{\mathscr{P}(m)}}} \ll_q \frac{1}{q^{n (d_{\mathscr{P}(m)}-1)} 2^{\Pi_q(m)}}.
	\end{equation}
	
	By \eqref{Equation_sum_Snepsilon}, the quadratic reciprocity law, and $(-1)^{\frac{(q-1)n r d_P}{2}} = (-1)^{\frac{(q-1)n d_P}{2}}$ for $r \geq 1$ odd,
	\begin{equation}\label{Sum_Characters}
		\sum_{Q \in S(n, m, (\delta)_m)} \sum_{h \in \calM_{\leq \rho(n,m)}} \frac{\chi_Q(h) \Lambda(h)}{N(h)} = \frac{1}{2^{\Pi_q(m)}} \sum_{f \mid \mathscr{P}(m)} \delta_f \sum_{\substack{P \in \calP, r \geq 1 \\ r d_P\leq \rho(n,m)} } (-1)^{\frac{(q-1)n d_P}{2}} \frac{ \log(q^{d_P})}{(q^{d_P})^r} \sum_{Q \in \calP_n} \left( \frac{P^r f}{Q} \right).
	\end{equation}

	We now separately compute the contribution for squares and non-squares $P^rf$'s in \eqref{Sum_Characters}. Since $\mathscr{P}(m)$ is square-free, $f$ is too. Therefore, $P^r f$ is a square if and only if $f = P$ and $r$ is odd. Thus, the contribution of the squares is
	\begin{equation}\label{Equation_Square_Contribution_Main_Term}
		\frac{\pi_q(n)}{2^{\Pi_q(m)}} \sum_{P \in \calP_{\leq m}} (-1)^{\frac{(q-1)n d_P}{2}} \delta_P \log(q^{d_P}) \sum_{\substack{r \geq 1 \text{ odd} \\ r d_P \leq \rho(n,m)}} \frac{1}{(q^{d_P})^r}.
	\end{equation}
	Now, observe that for any $P \in \calP_{\leq m}$ we have
	$$
	\sum_{\substack{r \geq 1 \text{ odd}\\ r d_P > \rho(n,m)}} \frac{1}{(q^{d_P})^r} \leq \sum_{r = \left\lfloor \frac{\rho(n,m)}{d_P}+1 \right\rfloor}^\infty \frac{1}{(q^{d_P})^r} \ll_q \frac{1}{q^{\rho(n,m)+d_P}}.
	$$
	Therefore, using \eqref{Prime_Number_Theorem} and \eqref{Equation_Sum_risone}, we can rewrite \eqref{Equation_Square_Contribution_Main_Term} as
	\begin{align*}
		&\frac{\pi_q(n)}{2^{\Pi_q(m)}}\left( 
		\sum_{\substack{P \in \calP_{\leq m} \\ r \geq 1 \text{ odd}}} (-1)^{\frac{(q-1)n d_P}{2}} \delta_P \frac{\log(q^{d_P})}{(q^{d_P})^r}   + O_q\left( \frac{1}{q^{\rho(n,m)}}\sum_{P \in \calP_{\leq m}}\frac{d_P}{q^{d_P}}\right) \right)\nonumber\\
		&= \frac{\pi_q(n)}{2^{\Pi_q(m)}}
		\sum_{\substack{P \in \calP_{\leq m} \\ r \geq 1 \text{ odd}}} (-1)^{\frac{(q-1)n d_P}{2}} \delta_P  \frac{\log(q^{d_P})}{(q^{d_P})^r} + O_q\left( \frac{m}{n} \cdot \frac{1}{2^{\Pi_q(m)}q^{n(2 d_{\mathscr{P}(m)}-1)}} \right).\numberthis\label{Error_Omega_Squares}
	\end{align*}
	
Now, suppose $P^r f$ is not a square. The same argument as in \cite{lumley_2019}*{p.269} yields
	$$
	\sum_{Q \in \calP_n} \left( \frac{ P^r f}{Q} \right) \ll (r d_P + d_f - 1) q^{n/2} \ll_q \rho(n,m) q^{n/2}.
	$$
	From \eqref{Prime_Number_Theorem}, $\rho(n,m) \leq 2 n \sum_{j=1}^m j \pi_q(j) \ll_q n q^m$. This, together with a computation similar to Lemma  \ref{Lemma_Sum_allr}, imply that the contribution of the non-squares is
	\begin{equation}\label{Error_Term_Odd_Contribution_Omega}
	\ll_q \frac{\rho(n,m) q^{\frac{n}{2}}}{2^{\Pi_q(m)}} \sum_{f \mid \calP(m)} \sum_{\substack{P \in \calP, r \geq 1 \\ r d_P \leq \rho(n,m)}} \frac{d_P}{\left( q^{d_P} \right)^r}	\ll_q n^2 q^{\frac{n}{2} + 2m}.
	\end{equation}
The result follows by combining \eqref{Ihara_Approx}, \eqref{Bound_Error_Ihara}, \eqref{Sum_Characters}, \eqref{Error_Omega_Squares} and \eqref{Error_Term_Odd_Contribution_Omega}.
\end{proof}

We finally come to the proof of Theorem \ref{Theorem_Omega_Result}.

\begin{proof}[~Proof of Theorem \ref{Theorem_Omega_Result}]
	Let $\omega \in \{-1,1\}$. For each real number $\varepsilon > 0$, set
	$$
	G_{\varepsilon,\omega} := \left\{ Q \in S(n, m, (\delta)_m) : \omega \frac{L'(1,\chi_Q)}{L(1,\chi_Q)} \geq -\omega\sum_{\substack{P \in \calP_{\leq m} \\ r \geq 1 \text{ odd}}} (-1)^{\frac{(q-1) n d_P}{2}} \delta_P \frac{\log(q^{d_P})}{\left(q^{d_P}\right)^r} - \varepsilon \right\}.
	$$
	Consider
	\begin{equation}\label{Equation_Decomposition}
	\sum_{Q \in S(n, m, (\delta)_m)} \frac{L'(1,\chi_Q)}{L(1,\chi_Q)} = \sum_{Q \in G_{\varepsilon,\omega}} \frac{L'(1,\chi_Q)}{L(1,\chi_Q)} + \sum_{Q \notin G_{\varepsilon,\omega}} \frac{L'(1,\chi_Q)}{L(1,\chi_Q)}.
	\end{equation}
	For any integer $n \geq 1$, it follows from \eqref{Lemma_Ihara_Upper_Bound} that there is a constant $\theta_q > 0$ such that
	\begin{equation}\label{Equation_Ihara}
		\omega\frac{L'(1,\chi_Q)}{L(1,\chi_Q)} <  \theta_q\log(\log(q^n)).
	\end{equation}
	
	If $Q \notin G_{\varepsilon,\omega}$ we have
	\begin{equation}\label{Equation_Outside}
		\omega\frac{L'(1,\chi_Q)}{L(1,\chi_Q)} < -\omega\sum_{\substack{P \in \calP_{\leq m} \\ r \geq 1 \text{ odd}}} (-1)^{\frac{(q-1)n d_P}{2} n d_P} \delta_P \frac{\log(q^{d_P})}{\left(q^{d_P}\right)^r} - \varepsilon.
	\end{equation}
	Employing \eqref{Equation_Ihara}
	and \eqref{Equation_Outside} in \eqref{Equation_Decomposition}, we have
	\begin{align*}
		\sum_{Q \in S(n, m, (\delta)_m)} \left(\omega\frac{L'(1,\chi_Q)}{L(1,\chi_Q)}\right)
		&< \theta_q\log(\log(q^n)) \sum_{Q \in G_{\varepsilon,\omega}} 1\\
		&+ \left(-\omega\sum_{\substack{P \in \calP_{\leq m} \\ r \geq 1 \text{ odd}}} (-1)^{\frac{(q-1)n d_P}{2} n d_P} \delta_P \frac{\log(q^{d_P})}{\left(q^{d_P}\right)^r} - \varepsilon\right) \sum_{Q \notin G_{\varepsilon,\omega}} 1.
	\end{align*}
	
	For each $P \in \calP_{\leq m}$, choose $\delta_P = \omega (-1)^{\frac{(q-1)n d_P}{2}+1}$. Then, for $\varepsilon > 0$ small enough, we have
	$$
	\sum_{Q \in S(n, m, (\delta)_m)} \left(\omega\frac{L'(1,\chi_Q)}{L(1,\chi_Q)}\right)  < \theta_q\log(\log(q^n)) \sum_{Q \in G_{\varepsilon,\omega}} 1 + \left( \sum_{\substack{P \in \calP_{\leq m} \\ r \geq 1 \text{ odd}}}  \frac{\log(q^{d_P})}{\left( q^{d_P} \right)^r} - \varepsilon \right) \#S(n, m, (\delta)_m) 
	$$
	and so
	\begin{align*}
		\sum_{Q \in G_{\varepsilon,\omega}} 1 &> \frac{1}{\theta_q\log(\log(q^n))}\left( \sum_{Q \in S(n, m, (\delta)_m)} \left(\omega \frac{L'(1,\chi_Q)}{L(1,\chi_Q)}\right) - \sum_{\substack{P \in \calP_{\leq m} \\ r \geq 1 \text{ odd}}}  \frac{\log(q^{d_P})}{\left(q^{ d_P}\right)^r} \#S(n, m, (\delta)_m) \right)\numberthis\label{Latter_Inequality}\\
		&+ \frac{\varepsilon}{\theta_q\log(\log(q^n))}\#S(n, m, (\delta)_m).
	\end{align*}
	Assuming $n \geq q^{m^{1/2}}$, we successively apply Proposition \ref{Proposition_Intermediate_Omega_Step}, \eqref{Equation_Cardinality_Snm} and \eqref{Prime_Number_Theorem} in \eqref{Latter_Inequality} and get
	\begin{align*}
		\sum_{Q \in G_{\varepsilon,\omega}} 1
		&> \frac{\varepsilon }{\theta_q n\log(\log(q^n))}\frac{q^n}{2^{\Pi_q(m)}} +
		O_{q,\varepsilon}\left(\frac{1}{n\log(\log(q^n))}\frac{q^\frac{n}{2}}{2^{\Pi_q(m)}} \sum_{\substack{P \in \calP_{\leq m} \\ r \geq 1 \text{ odd}}}  \frac{\log(q^{d_P})}{\left( q^{d_P}\right)^r}\right) \\
		&+ O_{q,\varepsilon}\left( \frac{n^2 q^{\frac{n}{2} + 2m}}{\log(\log(q^n))}\right) + O_{q,\varepsilon}\left( \frac{q^{\frac{n}{2}+m}}{\log(\log(q^n))} \sum_{\substack{P \in \calP_{\leq m} \\ r \geq 1 \text{ odd}}} \frac{\log(q^{d_P})}{\left( q^{d_P}\right)^r} \right) + O_{q,\varepsilon}\left( \frac{ q^{\frac{n}{2}+m}}{\log(\log(q^n))} \right).
	\end{align*}
	Then using Lemma \ref{Lemma_Sum_allr}, we find
	\begin{equation}\label{Bounding_Below_Sum_of_1}
	\sum_{Q \in G_{\varepsilon,\omega}} 1 > \frac{\varepsilon \log(q)}{\theta_q\log(q^n)\log(\log(q^n))}\frac{q^n}{2^{\Pi_q(m)}} + O_{q,\varepsilon}\left( \frac{q^{\frac{n}{2}+2m}\log^2(q^n) }{\log(\log(q^n))}\right).
	\end{equation}
	
	We now suppose that
	\begin{equation}\label{Equation_Conditions_on_m}
		q^m \leq  \log(q^n) \text{ and } 2^{\Pi_q(m)} \leq \left( q^n \right)^{\frac{1}{2}-5\varepsilon}.
	\end{equation}
	Assuming $m$ large, the first condition implies that $n \geq q^{m^{1/2}}$. Since $\log(q^n) \ll_\varepsilon (q^n)^\varepsilon$ for any $\varepsilon > 0$, then, $q^{2m} 2^{\Pi_q(m)} \log^2(q^n)  \leq \left( q^n \right)^{1/2-5\varepsilon} \log^4(q^n) \ll_{q,\varepsilon} \left( q^n \right)^{1/2-\varepsilon}$, and so
	\begin{equation}\label{Equation_Uniformizing_Error_Term}
		\frac{q^{\frac{n}{2} + 2m}\log^2(q^n)}{\log(\log(q^n))} \ll_{q,\varepsilon} \frac{\left( q^n \right)^{1-\varepsilon}}{2^{\Pi_q(m)}}.
	\end{equation}
	Hence, employing the second assumption of \eqref{Equation_Conditions_on_m} and  \eqref{Equation_Uniformizing_Error_Term} in \eqref{Bounding_Below_Sum_of_1}, we can write
	\begin{equation}\label{Equation_Counting_Gdeltomega}
	\sum_{Q \in G_{\varepsilon,\omega}} 1 > \frac{\varepsilon \log(q)}{\theta_q\log(q^n)\log(\log(q^n))} \frac{q^n}{2^{\Pi_q(m)}} + O_{q,\varepsilon}\left( \frac{\left( q^n \right)^{1-\varepsilon}}{2^{\Pi_q(m)}} \right) \gg_{q,\varepsilon} \left( q^n \right)^{\frac{1}{2} + 4\varepsilon}.
	\end{equation}
	
	From \cite{lumley_2019}*{Lemma 2.1}, given $\varepsilon > 0$, we can choose $m \geq 1$ large enough such that
	\begin{equation}\label{Equation_degree_up_to_m}
		\Pi_q(m) \leq \frac{\zeta_q(2)(1 + \varepsilon)q^m}{\log(q^m)}.
	\end{equation}
	In light of \eqref{Equation_degree_up_to_m}, the condition $2^{\Pi_q(m)} \leq \left( q^n \right)^{\frac{1}{2}-5\varepsilon}$ of \eqref{Equation_Conditions_on_m} holds if we choose any large enough integer $m \geq 1$ satisfying
	$$
	\frac{q^m}{\log(q^m)} \leq  \frac{(1-10\varepsilon)\log(q^n)}{2\log(2)\zeta_q(2)(1+\varepsilon)}.
	$$
	
	Choosing $0 < \varepsilon < 1/20$ and then a constant $\alpha > 42\log(2)\zeta_q(2)$, guarantees there is an integer $m \geq 1$ such that
	\begin{equation}\label{Equation_Interval_for_qm}
		\frac{\log(q^n)\log(\log(q^n))}{q\left(2 \log(2) \zeta_q(2)(1+\varepsilon)+ \alpha \varepsilon\right)} < q^m < \frac{\log(q^n)\log(\log(q^n))}{2 \log(2) \zeta_q(2)(1+\varepsilon)+ \alpha \varepsilon}.
	\end{equation}
	
	Let $m \geq 1$ satisfying \eqref{Equation_Interval_for_qm}. For any $Q \in G_{\varepsilon,\omega}$ Lemma \ref{Lemma_Lower_Bound} gives
	\begin{equation}\label{Equation_Almost_Final_Step}
		\omega\frac{L'(1,\chi_Q)}{L(1,\chi_Q)} \geq \sum_{\substack{P \in \calP_{\leq m} \\ r \geq 1 \text{ odd}}} \frac{\log(q^{d_P})}{\left(q^{d_P}\right)^r} - \varepsilon
		\geq \sum_{P \in \calP_{\leq m}} \frac{\log(q^{d_P})}{q^{d_P}}  - \varepsilon
		> (m - 2.61)\log(q) - \varepsilon.
	\end{equation}
	Note that the lower bound of \eqref{Equation_Almost_Final_Step} is positive since $0 < \varepsilon < 1/20 < 0.39$. Applying \eqref{Equation_Interval_for_qm} to \eqref{Equation_Almost_Final_Step} gives
	\begin{align*}
		\omega\frac{L'(1,\chi_Q)}{L(1,\chi_Q)}	&> \left(m - 2.61\right)\log(q) - \varepsilon\\
		&> \log(\log(q^n)) + \log(\log(\log(q^n))) - 2.61\log(q)\\
		&-\log\left(2\log(2)\zeta_q(2)(1+\varepsilon)\right) - \log\left( 1 + \frac{\alpha \varepsilon}{2\log(2)\zeta_q(2)(1+\varepsilon)} \right).
	\end{align*}
	Since $\log(1+x) \leq x$ if $x > -1$, then
	\begin{align*}
	\omega\frac{L'(1,\chi_Q)}{L(1,\chi_Q)}	
		&> \log(\log(q^n)) + \log(\log(\log(q^n))) - 2.61\log(q) - \log(2\log(2)\zeta_q(2))\numberthis\label{Equation_Last_Inequality}\\
		&-\log(1+\varepsilon) - \frac{\alpha \varepsilon}{2\log(2)\zeta_q(2)(1+\varepsilon)}.
	\end{align*}
	
	From \eqref{Equation_Counting_Gdeltomega}, \eqref{Equation_Last_Inequality} and specializing $\omega$ to $-1$ and $1$, we conclude the proof.
\end{proof}

\textbf{Acknowledgment}

F\'elix Baril Boudreau thanks Hector Pasten for asking about the geometric meaning of $L'(1,\chi_D)/L(1,\chi_D)$ during the 2023 CMS Winter Meeting in Montr\'eal, Dinesh Thakur for useful conversations on the Chowla-Selberg formula and Taguchi heights of Drinfeld modules, and Fabien Pazuki for mentioning their result \cite{Angles_Armana_Bosser_Pazuki_2024}*{Proposition 4.18} on the logarithmic Weil height of the j-invariant of Drinfeld modules. The authors also thank Alexandra Florea for answering questions on her work, Youness Lamzouri for the encouragement to work out  omega results and Chantal David for useful feedback on the paper.

\end{document}